\documentclass[a4paper,12pt]{amsart}

\usepackage{amssymb,color}

\numberwithin{equation}{section}

\newtheorem{thm}{Theorem}[section]
\newtheorem{prop}[thm]{Proposition}
\newtheorem{lemma}[thm]{Lemma}
\newtheorem{cor}[thm]{Corollary}

\theoremstyle{definition}
\newtheorem{defn}[thm]{Definition}

\theoremstyle{remark}
\newtheorem{rmk}[thm]{Remark}
\newtheorem{ex}[thm]{Example}

\newcommand{\C}{\mathbb{C}}
\newcommand{\N}{\mathbb{N}}
\newcommand{\T}{\mathbb{T}}
\newcommand{\Z}{\mathbb{Z}}
\newcommand{\F}{\mathbb{F}}

\newcommand{\tL}{\tilde{\Lambda}}

\newcommand{\lsp}{\operatorname{\mathrm{span}}}
\newcommand{\rank}{\operatorname{\mathrm{rank}}}

\DeclareMathOperator{\KP}{KP}
\DeclareMathOperator{\End}{End}

\newcommand{\red}{{\operatorname{\mathrm{red}}}}
\newcommand{\blue}{{\operatorname{\mathrm{blue}}}}
\def\Ff{\mathcal{F}}
\def\cycle{{\operatorname{\mathrm{cycle}}}}
\def\cycle{{{i}}}

\title[Kumjian-Pask algebras of locally convex higher-rank graphs]{Kumjian-Pask algebras \\ of locally convex higher-rank graphs}

\author{Lisa Orloff Clark}
\author{Claire Flynn}
\author{Astrid an Huef}
\address{Lisa Orloff Clark, Claire Flynn, Astrid an Huef: Department of Mathematics and Statistics, 
University of Otago, P.O. Box 56, Dunedin 9054, New Zealand.}
\thanks{The first author was supported by a University of Otago research grant and the second 
author was supported by a grant from the Division of Sciences at the University of Otago.}
\date{18 April, 2013}

\begin{document}

\begin{abstract} 
\end{abstract}

\subjclass[2000]{16W50}

\keywords{Kumjian-Pask algebra, higher-rank graph, Leavitt path algebra, Cuntz-Krieger algebra, graph algebra}

\begin{abstract}  The Kumjian-Pask algebra of a higher-rank graph generalises the Leavitt path algebra of a directed 
graph.  We extend the definition of Kumjian-Pask algebra to row-finite higher-rank graphs $\Lambda$ with sources which satisfy a 
local-convexity condition. After proving  versions of the graded-uniqueness theorem and the Cuntz-Krieger uniqueness theorem,
we study the Kumjian-Pask algebra of rank-$2$ Bratteli diagrams by studying certain finite subgraphs which are locally convex. 
 We show that the desourcification procedure of Farthing and Webster yields a  row-finite higher-rank graph $\tL$  without sources
 such that the Kumjian-Pask algebras of $\tL$ and $\Lambda$ are Morita equivalent.  We then use the Morita equivalence to study 
the ideal structure of the Kumjian-Pask algebra of $\Lambda$  by pulling the appropriate results across the equivalence.
\end{abstract}

\maketitle

\section{Introduction}
The Kumjian-Pask algebras were introduced in \cite{ACaHR} as higher-rank analogues of
the Leavitt path algebras associated to a directed graph $E$. Since they were introduced in \cite{AA1} and \cite{AMP}, 
the Leavitt path algebras have attracted a lot of attention,  see, for example, \cite{AA2, AA3, AB, AG, AGP, G, GR, Smith, T1, T}.  
A $k$-graph is a category $\Lambda$ with a degree functor $d:\Lambda\to\N^k$ which generalise the path category of $E$ and the 
length function $\lambda\mapsto|\lambda|$ on the paths $\lambda$ of $E$, respectively. Thus we think of the set $\Lambda^0$ of 
objects as vertices and of the morphisms  $\lambda\in\Lambda$ as paths of ``shape'' or degree $d(\lambda)$, and demand that paths 
factor uniquely: if $d(\lambda)=m+n$ in $\N^k$, then there exist unique $\mu,\nu\in\Lambda$ with $d(\mu)=m$ and $d(\nu)=n$ such 
that $\lambda=\mu\nu$.

Let $R$ be a commutative ring with $1$ and $\Lambda$ a row-finite $k$-graph without sources. The authors of \cite{ACaHR} construct 
a graded algebra $\KP_R(\Lambda)$, called the Kumjian-Pask algebra, which is universal for so-called Kumjian-Pask families.  
If $k=1$ then $\KP_R(\Lambda)$ is isomorphic to the Leavitt path algebra $L_R(E)$, 
where $E$ is the directed graph with vertices the objects of $\Lambda$ and edges the paths of degree 1.

Here we define a Kumjian-Pask algebra for row-finite $k$-graphs which may have sources but are ``locally convex'', a condition 
which restricts the types of sources that can occur.  We were motivated by the  study of $C^*$-algebras and Kumjian-Pask  
algebras of ``rank-$2$ Bratteli diagrams'' in \cite{PRRS, ACaHR}.  These Bratteli diagrams are $2$-graphs $\Lambda$ without 
sources, 
but their Kumjian-Pask algebras  can  profitably be studied by looking at finite subgraphs $\Lambda_N$ which have sources but are 
locally convex.  In \cite{ACaHR}, $\KP_{\C}(\Lambda)$ was analysed by embedding it inside  $C^*(\Lambda)$, and then using the 
$C^*$-subalgebras $C^*(\Lambda_N)$ of the subgraphs to deduce results about $\KP_\C(\Lambda)$. A sample theorem obtained 
in this way illustrates how the dichotomy for Leavitt path algebras of \cite[Theorem~4.5]{AA3} does not hold for Kumjian-Pask algebras.  
In particular, \cite[Theorem~7.10]{ACaHR} gives a class $\mathcal C$ of  rank-$2$ Bratteli diagrams such that 
for each $\Lambda \in \mathcal C$,  $\KP_\C(\Lambda)$ is 
simple but is neither purely infinite nor locally matricial.  A key performance indicator for our new definition of the
 Kumjian-Pask algebra was an extension of \cite[Theorem~7.10]{ACaHR} to arbitrary fields, and this is achieved in 
Theorem~\ref{thm-point} below.   For more motivation,  the K-theory of the $C^*$-algebras of these rank-$2$ Bratteli 
diagrams was computed in \cite{PRRS} as a direct limit of the K-groups of suitable $C^*(\Lambda_N)$, and a similar 
approach should work to compute the algebraic  K-theory of $\KP_R(\Lambda)$.

Now let $\Lambda$ be a locally convex, row-finite $k$-graph.   After finding the appropriate notion of Kumjian-Pask 
family of $\Lambda$ in this setting, we have obtained a new Kumjian-Pask algebra $\KP_R(\Lambda)$ with a very 
satisfactory theory: $\KP_R(\Lambda)$ is generated by a universal
Kumjian-Pask family $(p,s)$ and the properties of $(p,s)$ ensure that $\KP_R(\Lambda)=\lsp\{s_\lambda s_{\mu^*}:\lambda,\mu\in \Lambda\}$ (see \S\ref{sec-family}),   
and there are versions of both the graded-uniqueness and the Cuntz-Krieger uniqueness theorems (see \S\ref{sec-uniquenessthms}).  
Then, several applications of the graded-uniqueness theorem shows that for $\Lambda$ in the class $\mathcal C$ mentioned above, $\KP_R(\Lambda)$ is neither purely infinite nor locally matricial (see \S\ref{sec-Bratteli}).

There is a construction  by Farthing \cite{F} and Webster \cite{W}, called the desourcification of $\Lambda$,
 which yields a row-finite $k$-graph $\tL$ without sources such that the $C^*$-algebras $C^*(\tL)$ and $C^*(\Lambda)$
 are Morita equivalent.  In \S\ref{sec-Morita}, we show that $\KP_R(\tL)$ and $\KP_R(\Lambda)$ are Morita equivalent
 as well.  This result is new even when $k=1$ and  $\Lambda$ is the path category of a row-finite directed graph.  
In \S\S\ref{sec-simplicity}-\ref{sec-ideals} we  study the ideal structure of $\KP_R(\Lambda)$ by pulling the 
relevant results for $\KP_R(\tL)$ from \cite{ACaHR} across the Morita equivalence. Thus we obtain graph-theoretic 
characterisations of basic simplicity and simplicity, and show that there is a lattice isomorphism between the 
graded basic ideals of $\KP_R(\Lambda)$ and the saturated, hereditary subsets of $\Lambda^0$.

\section{Preliminaries}
We write $\N$ for the set of non-negative integers. We view $\N$ as a category with one object.  Fix $k\in \N \setminus \{0\}$.  
We often write $n \in \N^k$ as $(n_1, \dots, n_k)$, and say $m \leq n $ in $\N^k$ if and only if $m_i \leq n_i$ for
all $1\leq i \leq k$.   We use  $e_i$ for the usual basis elements in $\N^k$, so that $e_i$ is $1$ in the $i$th coordinate and $0$ in the others.  We denote the join and meet in $\N^k$ by $\vee$ and $\wedge$ respectively.

A \emph{$k$-graph} $(\Lambda, d)$  is a countable category $\Lambda$
with a functor $d: \Lambda \to \N^k$, called the \emph{degree map},  satisfying the \emph{factorisation property}:
if $d(\lambda) = m + n$ for some $m,n\in\N^k$, then there exist unique  $\mu$ and $\nu$ in $\Lambda$ such that
$d(\mu)=m, d(\nu)=n$ and $\lambda =\mu\nu$.  In this case, we often write $\lambda(0,m)$ for $\mu$. 

We denote the set of objects in $\Lambda$ by $\Lambda^0$ and use the factorisation property to identify  the morphisms $d^{-1}(0)$ of degree $0$ and  $\Lambda^0$.  We write $r$ and $s$  for the domain and codomain maps from $\Lambda$ to $\Lambda^0$.  
The path category associated to a directed graph is a $1$-graph, and motivated by this we call $r$ and $s$ the range and source maps, and elements of $\Lambda$ and $\Lambda^0$ paths and vertices, respectively.
For  $v \in \Lambda^0$ and $m \in \N^k$, we write 
\begin{align*}\Lambda^m &:= \{\lambda \in \Lambda : d(\lambda) = m\} \quad\text{and}\quad \Lambda^{\neq 0}:= \Lambda \setminus \Lambda^0,\\
 v \Lambda &:= \{\lambda \in \Lambda : r(\lambda) = v\} \quad\text{(this is denoted  $\Lambda(v)$ in \cite{RSY03})},\\
v \Lambda^m &:= \Lambda^m \cap v\Lambda.
\end{align*}

\begin{ex}
Fix   $m \in (\N \cup \{\infty\})^k$ and define 
 \[\Omega_{k,m}:= \{(p,q) \in \N^k \times \N^k :  p \leq q \leq m \}.
 \]
This is a category with objects
\[\Omega_{k,m}^0 =\{p \in \N^k \mid p \leq m\},\] 
and range and source maps $r(p,q)= p$ and $s(p,q) =q$.
Paths $(p,q)$ and $(r,s)$ are
composable if and only if $q=r$, and then $(p,q)(q,s) = (p,s)$.
With
$d:\Omega_{k,m} \to \N^k$ defined by $d((p,q)) = q-p$, the pair $(\Omega_{k,m}, d)$
is  a $k$-graph.  
We  write  $\Omega_k$ when  $m$ is infinite in every coordinate.
\end{ex}

Let $\Lambda$ be a $k$-graph.  
Then $\Lambda$ is \emph{row-finite} if $v\Lambda^n$ is finite for every $v\in\Lambda^0$ and $n\in\N^k$.   
A vertex $v\in\Lambda^0$ is a \emph{source} if there exists $n\in\N^k$ such that $v\Lambda^n=\emptyset$, 
that is, $v$ receives no paths of degree $n$.    We say  $\Lambda$ is \emph{locally convex}\label{lc} if for every $v\in \Lambda^0$, $1 \leq i,j \leq k$ with $i \neq j$,  $ \lambda \in v\Lambda^{e_i}$ and $\mu \in v\Lambda^{e_j}$, the sets
$s(\lambda)\Lambda^{e_j}$ and $s(\mu)\Lambda^{e_i}$ are nonempty \cite[Definition~3.10]{RSY03}.
Thus if $\Lambda$ has no sources, then $\Lambda$ is locally convex.  In this paper we only consider locally convex, row-finite $k$-graphs.

\subsection*{Paths, infinite paths and boundary paths}    
 The Cuntz-Krieger relation (KP4) for $k$-graphs without sources (see Section~\ref{sec-family}) involves the 
sets $v\Lambda^n$ of paths of degree $n$ with 
range $v$. When $\Lambda$ has sources, $v\Lambda^n$ could be empty.  The technical innovation in \cite{RSY03}
 is to introduce the set $\Lambda^{\leq n}$ consisting of paths $\lambda$ with $d(\lambda)\leq n$ which cannot be extended to paths $\lambda\mu$ with $d(\lambda)<d(\lambda\mu)\leq n$. Thus
\begin{equation*}
\Lambda^{\leq n}:=\{\lambda\in\Lambda: d(\lambda)\leq n, \text{\ and\ }d(\lambda)_i<n_i \text{\ implies\ }s(\lambda)\Lambda^{e_i}=\emptyset\},
\end{equation*}
 and then  $v\Lambda^{\leq n} := v\Lambda \cap \Lambda^{\leq n}$ for $v \in \Lambda^0$
is always nonempty.
For example, if $n=e_i$ for some $1\leq i \leq k$, then 
\[
v\Lambda^{\leq e_i} = \begin{cases} v\Lambda^{e_i} & \text{ if }v\Lambda^{e_i} \neq \emptyset;\\
 \{v\} & \text{ otherwise.}\end{cases}
\]

A $k$-graph morphism is a degree-preserving functor.  An \emph{infinite path} in a $k$-graph $\Lambda$ is a $k$-graph morphism $x:\Omega_k\to \Lambda$.
In a graph with sources, not every finite path is contained in an infinite path, and 
another technical innovation of \cite{RSY03} is to replace the space $\Lambda^\infty$
of infinite paths with a space of so-called  boundary paths.

Let $\Lambda$ be  a locally convex, row-finite $k$-graph and  $m \in (\N \cup \{\infty\})^k$.  In Definition~3.14 of \cite{RSY03}, a graph morphism
$x: \Omega_{k,m} \to \Lambda$ is defined to be a \emph{boundary path} of degree $m$  
if 
\begin{equation}
 \label{eq:boundarypath}
v \in  \Omega_{k,m}^0 \text{ and }  v\Omega_{k,m}^{\leq e_i} = \{v\} \text{ imply }
x(v)\Lambda^{\leq e_i} = \{x(v)\}.
\end{equation}  
Thus a  boundary path  maps sources to sources, and every infinite path is a boundary path.  We denote the set of boundary paths by $\Lambda^{\leq \infty}$. If $\Lambda$ has no sources, then $\Lambda^{\leq\infty}=\Lambda^\infty$. 
Since we identify the object $n\in \Omega_{k,m}$ with the identity morphism $(n,n)$ at $n$, we write $x(n)$ for the vertex $x(n,n)$. Then the range of a boundary path $x$ is the vertex $r(x):=x(0)$. We set  $v\Lambda^{\leq \infty}:=\Lambda^{\leq \infty}\cap r^{-1}(v)$ and $v\Lambda^{\infty}:=\Lambda^{\infty}\cap r^{-1}(v)$.  A boundary path $x$ is completely determined by the set of paths $\{x(0,n): n\leq d(x)\}$,  hence can be composed with finite paths, and there is a converse factorisation property.  We denote by $\sigma$ the partially-defined shift map on $\Lambda^{\leq \infty}$ which
 is defined by $\sigma^m(x)=x(m,\infty)$ for $x\in \Lambda^{\leq \infty}$ when $m\leq d(x)$.

\section{Kumjian-Pask $\Lambda$-families}\label{sec-family} Throughout this section, $\Lambda$ is a row-finite $k$-graph 
and $R$ is a commutative ring with $1$.
Define  $G(\Lambda):= \{\lambda^* : \lambda \in \Lambda\}$, and call each
$\lambda^*$  a \emph{ghost path}. If $v\in\Lambda^0$, then we identify $v$ and $v^*$. 
We extend the degree functor $d$ and the range and source maps $r$ and $s$ to $G(\Lambda)$ by
\[
 d(\lambda^*) = -d(\lambda),\quad r(\lambda^*) = s(\lambda) \quad\text{ and }\quad s(\lambda^*)=r(\lambda).
\]
We extend the factorisation property to the ghost paths by setting $(\mu\lambda)^* = \lambda^*\mu^*$.  We denote by $G(\Lambda^{\neq 0})$ the set of ghost paths that are not vertices.

Let $\Lambda$ be a row-finite $k$-graph without sources.  Recall from \cite[Definition~3.1]{ACaHR} that a
 \emph{Kumjian-Pask $\Lambda$-family $(P,S)$ \label{oldKPrelations}
in an $R$-algebra $A$} consists of two functions $P:\Lambda^0\to A$ and 
$S:\Lambda^{\not=0}\cup G(\Lambda^{\not=0})\to A$ such that\label{KPwithoutsources}
\begin{enumerate}
\item[(KP1)] $\{P_v:v\in \Lambda^0\}$ is an  orthogonal set of idempotents in the sense 
that $P_vP_w=\delta_{v,w}P_v$,
\item[(KP2)] for all $\lambda, \mu\in\Lambda^{\neq 0}$ with $r(\mu) = s(\lambda)$, we have
\[
S_{\lambda}S_{\mu} = S_{\lambda\mu}, \; S_{\mu^*}S_{\lambda^*} = S_{(\lambda\mu)^*}, \;
 P_{r(\lambda)}S_{\lambda} = S_{\lambda} = S_{\lambda}P_{s(\lambda)}, \;
  P_{s(\lambda)}S_{\lambda^*} = S_{\lambda^*} = S_{\lambda^*}P_{r(\lambda)},
\]
\item[(KP3)] for all $\lambda, \mu \in\Lambda^{\neq 0}$ with $d(\lambda) = d(\mu)$, we have
\[
S_{\lambda^*}S_{\mu} = \delta_{\lambda,\mu}P_{s(\lambda)},
\]
\item[(KP4)] for all $v\in\Lambda^0$ and all $n\in {\mathbb N}^k\setminus \{0\}$, we have
\[
P_v = \sum_{\lambda\in v\Lambda^n} S_{\lambda}S_{\lambda^*}.
\]
\end{enumerate} 

The sum in (KP4) is finite because $\Lambda$ is row-finite. 
The  relations (KP1)--(KP4) were obtained in \cite[\S3]{ACaHR} by adding to the usual 
Cuntz-Krieger relations from \cite{KumPas00}.  A Cuntz-Krieger $\Lambda$-family $(P,S)$  
in the algebra $B(H)$  of linear bounded operators on a Hilbert space $H$ consists of an
 orthogonal  set $\{P_v\}$ of projections  and a set  $\{S_\lambda\}$  of partial isometries 
$S_\lambda$ with initial projection $P_{s(\lambda)}$, satisfying weaker versions of  (KP1)--(KP4).
The geometric structure of $B(H)$ is rich enough  to yield (KP1)--(KP4) as given above, where 
$S_{\lambda^*}$  is the Hilbert space adjoint of $S_\lambda$. For example, the Cuntz-Krieger relation 
corresponding to 
(KP2) is just $S_{\lambda}S_{\mu}= S_{\lambda\mu}$; the rest of (KP2) comes for free.   
See \S3 of \cite{ACaHR} for more detail. \label{page5}

An important consequence of the  Kumjian-Pask relations is that the algebra generated by a 
Kumjian-Pask $\Lambda$-family $(P,S)$ is $\lsp_R\{S_\lambda S_{\mu^*}:\lambda,\mu\in \Lambda\}$, 
where we use the convention that $S_v:=P_v$ and $S_{v^*}:=P_{v}$ for $v\in \Lambda^0$. When  $\Lambda$
 has no sources, this follows from \cite[Lemma~3.3]{ACaHR}, which says that if $n\geq d(\lambda), d(\mu)$, then 
\begin{equation}\label{eq-span}
S_{\lambda^*}S_{\mu} = \sum_{d(\lambda\alpha)=n,\;\lambda\alpha = \mu\beta}S_{\alpha}S_{\beta^*}.
\end{equation}
Definition~\ref{KPsources} below gives a notion of Kumjian-Pask family which applies to $k$-graphs 
$\Lambda$ with sources. It is based on the approach by Raeburn, Sims and Yeend in \cite{RSY03} for 
Cuntz-Krieger $\Lambda$-families. The purpose of the new relations (KP3$'$) and (KP4$'$) is to ensure 
that we obtain a version of \eqref{eq-span}. 

\begin{defn}\label{KPsources}
Let $\Lambda$ be a row-finite $k$-graph (possibly with sources).
A Kumjian-Pask $\Lambda$-family $(P,S)$ 
in an $R$-algebra $A$ consists of two functions $P:\Lambda^0 \to A$ and 
$S:\Lambda^{\neq0}\cup G(\Lambda^{\neq 0}) \to A$ such that (KP1) and (KP2) hold, and 
\begin{enumerate}
\item[(KP3$'$)] for all $n \in \N^k\setminus\{0\}$ and 
$\lambda, \mu \in \Lambda^{\leq n}$, we have \[S_{\lambda^{*}}S_\mu=\delta_{\lambda,\mu}P_{s(\lambda)};\]
\item[(KP4$'$)] for all $v \in \Lambda^0$ and $n \in \N^k\setminus \{0\}$,  
\begin{equation*}
P_v=\sum_{\lambda \in v\Lambda^{\leq n}}S_\lambda S_{\lambda^{*}}.
\end{equation*}
\end{enumerate}
\end{defn}

Since (KP4$'$) is  the same as the fourth Cuntz-Krieger relation
of \cite[Definition~3.3]{RSY03},
we get the following.

\begin{lemma}[\text{\cite[Proposition~3.11]{RSY03}}]\label{lem-kp4alt}
Let $\Lambda$ be a locally convex, row-finite $k$-graph.  Then (KP4$'$) 
holds at $v\in\Lambda^0$ if and only if, for $1\leq i\leq k$ with
$v\Lambda^{e_i} \neq \emptyset$, $P_v=\sum_{\lambda \in v\Lambda^{e_i}}S_\lambda S_{\lambda^{*}}$.
\end{lemma}

The next lemma gives a version of  \eqref{eq-span}.

\begin{prop}
\label{prop:lspfamily}
Let $\Lambda$ be a locally convex, row-finite $k$-graph, $(P,S)$
a Kumjian-Pask $\Lambda$-family in an $R$-algebra $A$, and $\lambda,\mu \in \Lambda$.
If $n\in \N^k$  such that $d(\lambda),d(\mu) \leq n$, 
then 
\[
S_{\lambda^{*}}S_\mu=\sum_{\lambda\alpha=\mu\beta, 
\lambda\alpha \in \Lambda^{\leq n}} S_\alpha S_{\beta^{*}}.
\]
\end{prop}

\begin{proof}
Fix $n\in \N^k$  such that $d(\lambda),d(\mu) \leq n$. Then
\begin{align*}
S_{\lambda^{*}}S_\mu &= (P_{s(\lambda)}S_{\lambda^{*}})(S_\mu P_{s(\mu)})
 \quad  \text{ by (KP2)}\\
&=\Big( \sum_{\alpha \in s(\lambda)\Lambda^{\leq n-d(\lambda)}}
 S_\alpha S_{\alpha^{*}}\Big)S_{\lambda^{*}}S_\mu \Big( \sum_{\beta 
\in s(\mu)\Lambda^{\leq n-d(\mu)}} S_\beta S_{\beta^{*}}\Big) \quad \text{ by (KP4$'$)}\\
&= \sum_{\alpha \in s(\lambda)\Lambda^{\leq n-d(\lambda)}} \sum_{\beta \in s(\mu)\Lambda^{\leq n-d(\mu)}} 
S_\alpha S_{(\lambda\alpha)^{*}}
S_{\mu\beta} S_{\beta^{*}} \quad \text{ by (KP2)}
\\
&= \sum_{\alpha \in s(\lambda)\Lambda^{\leq n-d(\lambda)}} 
\sum_{\beta \in s(\mu)\Lambda^{\leq n-d(\mu)},\lambda\alpha=\mu\beta}S_\alpha P_{s(\mu\beta)} S_{\beta^{*}}
\intertext{by applying  (KP3$'$) to each summand. By unique factorisation, for each $\alpha$ there is just one $\beta$ of the given degree such that $\lambda\alpha=\mu\beta$, and the sums collapse to}
&= \sum_{\alpha \in s(\lambda)\Lambda^{\leq n-d(\lambda)}, \lambda\alpha=\mu\beta} 
S_\alpha  S_{\beta^{*}}
\end{align*}
This proves the lemma after noting that the purely graph-theoretic result \cite[Lemma~3.6]{RSY03} says that composing $\lambda$ with $\alpha\in s(\lambda)\Lambda^{\leq n-d(\lambda)}$ gives the path $\lambda\alpha\in\Lambda^{\leq n}$.
\end{proof}

\begin{cor}\label{cor-span} Let $\Lambda$ be a locally convex, row-finite $k$-graph and $(P,S)$
a Kumjian-Pask $\Lambda$-family in an $R$-algebra $A$. The subalgebra generated by $(P,S)$ is
$\lsp\{S_\alpha S_{\beta^*}:\alpha,\beta\in\Lambda, s(\alpha)=s(\beta)\}$.
\end{cor}
\begin{proof} We have $S_\alpha S_{\beta^*}=S_\alpha P_{s(\alpha)}P_{s(\beta)}S_{\beta^*}$
 by (KP2), so $S_\alpha S_{\beta^*}=0$ unless $s(\alpha)=s(\beta)$ by (KP1).  The result now 
follows from Proposition~\ref{prop:lspfamily} and (KP2).
\end{proof}

The set of \emph{minimal common  extensions of $\lambda,\mu\in\Lambda$} is
\[
\Lambda^{\min}(\lambda,\mu):=\{ (\alpha,\beta): \lambda\alpha=\mu\beta, d(\lambda\alpha)=d(\lambda)\vee d(\mu)\}.
\]
\label{Lambdamin}

\begin{cor}\label{cor-altKP3'}
Let $\Lambda$ be a locally convex, row-finite $k$-graph and $(P,S)$
a  family in an $R$-algebra $A$ satisfying (KP1), (KP2) and (KP4$'$). Then (KP3$'$) holds if and 
only if, for all  $\lambda,\mu \in \Lambda$, 
\begin{equation}\label{KP3'alt}
S_{\lambda^{*}}S_\mu=\sum_{(\alpha,\beta)\in\Lambda^{\min}(\lambda,\mu)}
S_\alpha S_{\beta^{*}}.
\end{equation}
\end{cor}

\begin{proof} Suppose (KP3$'$) holds. Then $(P,S)$ is a Kumjian-Pask $\Lambda$-family. 
Let $\lambda,\mu \in \Lambda$ and apply Proposition~\ref{prop:lspfamily} with 
$n=d(\lambda)\vee d(\mu)$ to get 
$S_{\lambda^{*}}S_\mu=\sum_{(\alpha,\beta)\in\Lambda^{\min}(\lambda,\mu)}
S_\alpha S_{\beta^{*}}$. 

Conversely, suppose that for all  $\lambda,\mu \in \Lambda$, \eqref{KP3'alt} holds. 
Fix $n\in\N^k\setminus\{0\}$ and let $\lambda,\mu\in\Lambda^{\leq n}$. Note 
that $d(\lambda)\vee d(\mu)\leq n$. First suppose that $d(\lambda)=d(\mu)$. Then
\[
\Lambda^{\min}(\lambda,\mu)=\begin{cases} 
\{(s(\lambda), s(\lambda))\}&\text{if $\lambda=\mu$;}
\\
\emptyset&\text{else,} \end{cases}
\] 
and \eqref{KP3'alt} gives $S_{\lambda^{*}}S_\mu=\delta_{\lambda,\mu}P_{s(\lambda)}$. 
Second, suppose $d(\lambda)\neq d(\mu)$. Then at least one of $\lambda,\mu$ has degree
 less than $n$, say $d(\lambda)<n$. But $\lambda\in\Lambda^{\leq n}$, and so there is
 no $\alpha$ such that $d(\lambda)<d(\lambda\alpha)\leq n$. Now
\begin{align*}
\Lambda^{\min}(\lambda,\mu)&=\{(s(\lambda),\beta):\lambda=\mu\beta, d(\lambda)=d(\lambda)\vee d(\mu)\}\\
&=\{(s(\lambda),\beta):\lambda=\mu\beta, d(\mu)< d(\lambda)\}
\end{align*}
since $d(\lambda)\neq d(\mu)$. But $\mu\in\Lambda^{\leq n}$ too, so $\Lambda^{\min}(\lambda,\mu)=\emptyset$. 
By \eqref{KP3'alt}
$S_{\lambda^*}S_\mu=0$, and $\lambda\neq \mu$ implies that $S_{\lambda^{*}}S_\mu=\delta_{\lambda,\mu}P_{s(\lambda)}$.
Thus in either case, $S_{\lambda^{*}}S_\mu=\delta_{\lambda,\mu}P_{s(\lambda)}$ 
as required.
\end{proof}

In order to demonstrate the existence of a nonzero Kumjian-Pask $\Lambda$-family 
we need to impose the
``local convexity'' condition from \cite{RSY03}.

\begin{prop}
 \label{prop:nonzerofam}
Let $\Lambda$ be a locally convex, row-finite $k$-graph.  Then there exist an $R$-algebra $A$ and a
Kumjian-Pask $\Lambda$-family $(P,S)$ in  $A$ with  $S_{\lambda}, S_{\lambda^*}, P_v  \neq 0$ for all $v\in \Lambda^0$ and $\lambda\in\Lambda$.  In particular, for every 
$r \in R\setminus\{0\}$ and $v\in\Lambda^0$, we have $rP_v \neq 0$.
\end{prop}

\begin{proof}
We modify the construction of the infinite-path representation of \cite{ACaHR} and instead
build a `boundary-path representation'.   Let $\F_R(\Lambda^{\leq \infty})$ be the free $R$-module
on $\Lambda^{\leq \infty}$.
For each $v \in \Lambda^0$ and $\lambda, \mu \in \Lambda^{\leq \infty}$,
define functions $f_v,f_{\lambda}$ and $f_{\mu}:\Lambda^{\leq \infty} \to \F_R(\Lambda^{\leq \infty})$ by
\begin{align*} f_v(x)&=
     \begin{cases}
      x & \text{if $r(x)=v$;}\\
      0 & \text{otherwise,}
     \end{cases}\\
f_\lambda(x)&=
     \begin{cases}
      \lambda x & \text{if $r(x)=s(\lambda)$;}\\
      0 & \text{ otherwise,}
     \end{cases}\\
f_{\mu^{*}}(x)&=
     \begin{cases}
      y & \text{if $x=\mu y$ for some $y\in\Lambda^{\leq \infty}$;}\\
      0 & \text{ otherwise.}
     \end{cases}
\end{align*}
By the universal property of free modules, there exist nonzero  
$P_v, S_{\lambda}, S_{\mu^{*}} \in \End(\F_R(\Lambda^{\leq \infty}))$ 
extending $f_v$, $f_{\lambda}$ and $f_{\mu^{*}}$. Note that $rP_v \neq 0$ for every $r \in R\setminus\{0\}$.

We claim that $(P,S)$ is a Kumjian-Pask $\Lambda$-family in
the $R$-algebra 
$\End(\F_R(\Lambda^{\leq \infty}))$.  Relations (KP1) and (KP2) are
straight-forward to check.  To see (KP3$'$), fix $n \in \N^k\setminus \{0\}$,
$\lambda, \mu \in \Lambda^{\leq n}$ and $x \in \Lambda^{\leq \infty}$. If $r(\mu)\neq r(\lambda)$, then both
 $S_{\lambda^*}S_{\mu}=S_{\lambda^*}P_{r(\lambda)}P_{r(\mu)}S_{\mu}$ and $\delta_{\lambda,\mu} P_{s(\lambda)}$ are $0$. So we may assume $r(\mu)=r(\lambda)$.
Notice that
\begin{equation*}
 S_{\lambda^*}S_{\mu}(x) 
= \begin{cases}
     S_{\lambda^*}(\mu x) & \text{if } x(0)=s(\mu) \text{ and $\mu x=\lambda y$ for some $y\in\Lambda^{\leq\infty}$}
;\\
0&\text{otherwise}.
   \end{cases}
\end{equation*}

Since $\lambda,\mu \in r(\lambda)\Lambda^{\leq n}$, 
$(\mu x)(0,d(\lambda))=\lambda$ implies
either $\lambda = \mu \lambda'$ or $\mu= \lambda \mu'$ for
some $\lambda',\mu' \in \Lambda$.  But then $\lambda = \mu$ by
the definition of $\Lambda^{\leq n}$.  Hence $(\mu x)(0,d(\lambda))=\lambda$ 
if and only if $\mu = \lambda$.
Thus 
\begin{align*}
 S_{\lambda^*}S_{\mu}(x) 
&= \begin{cases}
     x & \text{if } x(0)=s(\mu) \text{ and }
 \lambda = \mu;\\
0&\text{otherwise},
   \end{cases}
\\
&= \delta_{\lambda,\mu}P_{s(\lambda)}(x)
\end{align*}
and (KP3$'$) holds.

For (KP4$'$),  fix  $v\in\Lambda^0$ and $1\leq i\leq k$ with
$v\Lambda^{e_i} \neq \emptyset$. Since
$\Lambda$ is locally convex, it suffices to show that
 $P_v=\sum_{\lambda \in v\Lambda^{e_i}}S_\lambda S_{\lambda^{*}}$ 
by Lemma~\ref{lem-kp4alt}.   Let $x \in \Lambda^{\leq \infty}$.
Then
\begin{align*}
 \sum_{\lambda \in v\Lambda^{e_i}}S_{\lambda} S_{\lambda^{*}} (x)
= \sum_{\lambda \in v\Lambda^{e_i}} \delta_{\lambda,x(0,e_i)}x
&= \begin{cases}
    x& \text{if } r(x) =v;\\
0&\text{otherwise},
   \end{cases}\\
&= P_v(x).\qedhere
\end{align*}
\end{proof}

 We are now ready to show that there is an $R$-algebra which is ``universal for Kumjian-Pask 
$\Lambda$-families''; the proof is very similar to the one for $k$-graphs without sources \cite[Theorem~3.4]{ACaHR}, so we will just give an outline addressing the main points. 
  This $R$-algebra is graded over $\Z^k$, and to see  that the graded subgroups  have a nice description uses Proposition~\ref{prop:lspfamily}; so we will need to check carefully that the argument used when $\Lambda$ has no sources still works when $\Lambda^n$  is replaced by $\Lambda^{\leq n}$.  We will follow the convention of \cite{ACaHR} and use lower-case letters for universal Kumjian-Pask families.

\begin{thm}\label{thm-KP}
Let $\Lambda$ be a locally convex, row-finite $k$-graph.
\begin{enumerate}
\item\label{item-a}
There is an $R$-algebra $\KP_R (\Lambda)$, generated by a Kumjian-Pask 
$\Lambda$-family $(p,s)$, such that
if $(Q,T)$ is a Kumjian-Pask $\Lambda$-family in an $R$-algebra $A$, 
then there exists a unique $R$-algebra homomorphism $\pi_{Q,T}: \KP_R (\Lambda) \to A$ such that 
$\pi_{Q,T} \circ p=Q$ and $\pi_{Q,T} \circ s=T$.  For every 
$r \in R\setminus\{0\}$ and $v\in\Lambda^0$, we have $rp_v \neq 0$.
\item\label{item-b} The subsets
\[\KP_R(\Lambda)_n := \lsp\{s_\alpha s_{\beta^{*}}: d(\alpha)-d(\beta)=n\}\] 
form a $\Z^k$-grading of $\KP_R(\Lambda)$.
\end{enumerate} 
\end{thm}

\begin{proof}
Let $X:=\Lambda^0 \cup \Lambda^{\neq 0} \cup G(\Lambda^{\neq 0})$ and $\F_R(w(X))$ be the free algebra 
on the set $w(X)$ of words on $X$. Let $I$ be the ideal of $\F_R(w(X))$ generated by elements from the sets:
\begin{enumerate}
\item[(i)] $\{vw-\delta_{v,w}v: v,w \in \Lambda^0\}$,
\item[(ii)] $\{\lambda-\mu\nu, \lambda^{*} -\nu^{*}\mu^{*}: \lambda,\mu,\nu \in \Lambda^{\neq 0}
\text{ an } \lambda = \mu\nu\}$,
\item[(iii)] $\{\lambda-r(\lambda)\lambda ,\lambda - \lambda s(\lambda),
\lambda^{*}-s(\lambda)\lambda^{*},\lambda^{*}-\lambda^{*}r(\lambda)
:\lambda \in \Lambda^{\neq 0}\}$,
\item[(iv)] $\{\lambda^{*}\mu-\delta_{\lambda,\mu}s(\lambda):
\lambda, \mu \in \Lambda^{\leq n}, n \in \N^k \setminus \{0\}\}$,
\item[(v)] $\{v-\sum_{\lambda \in v\Lambda^{\leq n}} \lambda \lambda^{*}: 
v \in \Lambda^0 \text{ and }n \in \N^k \setminus \{0\}\}$.
\end{enumerate}
Set $\KP_R(\Lambda):=\F_R(w(X))/I$, and write $q: \F_R(w(X))\to \F_R(w(X))/I$ for the quotient map. Define $p:\Lambda^0 \to \KP_R(\Lambda)$ by $p_v=q(v)$, and  
$s: \Lambda^{\neq 0} \cup G(\Lambda^{\neq 0}) \to KP_R(\Lambda)$ by 
$s_\lambda=q(\lambda)$ and $s_{\lambda^{*}}=q(\lambda^{*})$. 
Then $(p,s)$ is a Kumjian-Pask $\Lambda$-family in the $R$-algebra $\KP_R(\Lambda)$. 

Now let $(Q,T)$ be a Kumjian-Pask $\Lambda$-family in an $R$-algebra $A$. 
Define $f:X \to A$ by $f(v)=Q_v$, $f(\lambda)=T_\lambda$ and $f(\lambda^{*})=T_{\lambda^{*}}$. 
The universal property of $\F_R(w(X))$  gives a unique $R$-algebra homomorphism $\psi: \F_R(w(X)) \to A$ such 
that $\psi|_X=f$. Since $(Q,T)$ is a Kumjian-Pask family,  $I \subseteq \ker(\psi)$. Thus there exists a unique  
$R$-algebra homomorphism $\pi_{Q,T}:\KP_R(\Lambda) \to A$  such that $\pi_{Q,T}\circ q=\psi$. It follows that 
$\pi_{Q,T} \circ p=Q$ and $\pi_{Q,T} \circ s=T$. 

Now fix $r \in R\setminus\{0\}$ and $v\in\Lambda^0$. If  $rp_v$ were zero then $r\pi_{Q,T}(p_v)=rQ_v$ would be 
zero for every Kumjian-Pask $\Lambda$-family $(Q,T)$.  But this is not the case for the Kumjian-Pask family of Proposition~\ref{prop:nonzerofam}.  
Thus $rp_v\neq 0$. This completes the proof of~\eqref{item-a}.

For~\eqref{item-b},  extend the degree map to   words on $X$ by setting $d:w(X) \to \Z^k$
 by $d(w)=\sum_{i=1}^{|w|} d(w_i)$.
By \cite[Proposition~2.7]{ACaHR},   $\F_R(w(X))$ is graded over $\Z^k$ by
the subgroups
\[\F_R(w(X))_n:=\Big\{ \sum_{w \in w(X)} r_w w: r_w \neq 0 \implies d(w)=n \Big\}.\]

We claim that the ideal $I$ defined in the proof of~\eqref{item-a} is graded.  For this, it suffices to see that $I$ is generated by homogeneous elements, that is, elements in $\F_R(w(X))_n$ for some $n\in \Z^k$.  The generators of $I$ in (i) are a linear combination of words of degree $0$, hence are homogeneous of degree $0$. If $\lambda=\mu\nu$ in $\Lambda$ then  $\lambda-\mu\nu$ is a linear combination of words of degree $d(\lambda)$, so  all the generators in (ii) are homogeneous. Also, $\lambda-r(\lambda)\lambda$ is homogeneous of degree $\lambda$, and similarly  all the generators in (iii) are homogeneous of some degree\footnote{This has caused some confusion before: for example in \cite[Proof of Proposition~4.7]{T}, $e-r(e)e$ for an edge $e$ in a graph is claimed to be $0$-graded whereas it is $1$-graded.}.  The elements in  (iv) are either of the form  $\lambda^*\lambda-s(\lambda)$ or of the form $\lambda^*\mu$; the former is homogeneous of degree $0$ and the latter is homogeneous of degree $d(\mu)-d(
\lambda)$. A word $\lambda\lambda^*$ has degree $0$, and hence the generators in (v) are homogeneous of degree $0$.
Thus $I$ is a graded ideal.

Since $I$ is graded,  the quotient $\KP_R(\Lambda)$ of
$\F_R(w(X))$ by $I$ is graded by the subgroups 
\[
(\F_R(w(X))/I)_n:=\lsp \{ q(w):w\in w(X),  d(w)=n \}.
\]
By Corollary~\ref{cor-span}, $\KP_R(\Lambda)=\lsp\{s_\alpha s_{\beta^*}:\alpha,\beta\in\Lambda, s(\alpha)=s(\beta)\}$. We need to show that 
\[
\KP_R(\Lambda)_n :=  \lsp\{s_\alpha s_{\beta^{*}}: d(\alpha)-d(\beta)=n\}=(\F_R(w(X))/I)_n.
\]
First, fix $s_\lambda s_{\mu^{*}} \in \{s_\alpha s_{\beta^{*}}:d(\alpha)-d(\beta)=n\}$. 
Then $s_\lambda s_{\mu^{*}}=q(\lambda)q(\mu^*)=q(\lambda\mu^{*})$,
 and  $d(\lambda\mu^{*})=d(\lambda)-d(\mu)=n$. Thus
$s_\lambda s_\mu^{*} \in \{q(w):d(w)=n, w \in w(X)\}\in (\F_R(w(X))/I)_n$. 

 That $ (\F_R(w(X))/I)_n \subseteq \KP_R(\Lambda)_n$ follows  immediately from the next lemma; the proof is very similar to that of \cite[Lemma~3.5]{ACaHR}, but we need to check replacing $\Lambda^n$  by $\Lambda^{\leq n}$ in the argument does not cause problems.  
\end{proof}

\begin{lemma}Let $X:=\Lambda^0 \cup \Lambda^{\neq 0} \cup G(\Lambda^{\neq 0})$ and  
$q:\F_R(w(X)) \to  \KP_R(\Lambda) $ be the quotient map.
If $w\in w(X)$, then $q(w)\in \KP_R(\Lambda)_{d(w)}$. 
\end{lemma} 

\begin{proof} The proof is by induction on $|w|$. We treat the cases $|w|=1,2$ separately. Recall that by our convention  $s_v:=p_v$ and $s_{v^*}:=p_v$ for $v\in \Lambda^0$.

 If $|w|=1$ there are two possibilities.   If $w=\lambda$ for some $\lambda\in \Lambda$, then $q(w)=s_\lambda=s_\lambda s_{s(\lambda)^*}$ and $d(\lambda)-d(s(\lambda))=d(\lambda)$, and so 
$q(w)\in \KP_R(\Lambda)_{d(w)}$.  Otherwise,  if $w=\lambda^*$,  then $q(w)=s_{\lambda^*}=s_{s(\lambda)}s_{\lambda^*}$ and $d(s(\lambda))-d(\lambda)=d(\lambda^*)$, so $q(w)\in \KP_R(\Lambda)_{d(w)}$. 

 If $|w|=2$ there are four possibilities: $w=\lambda\mu^*,  \lambda\mu, \mu^*\lambda^*$  or $\lambda^*\mu$. The first three possibilities are quickly dealt with since
 \begin{align*}
& q(\lambda\mu^*)=s_\lambda s_{\mu^*}\text{\ and \ } d(\lambda)-d(\mu)=d(\lambda\mu^*),\\
& q(\lambda\mu)=s_{\lambda\mu} s_{s(\mu)^*}\text{\ and \ } d(\lambda\mu)-d(s(\mu))=d(\lambda\mu),\\
& q(\mu^*\lambda^*)=s_{s(\mu)}s_{(\lambda\mu)^*}\text{\ and \ } d(s(\mu))-d((\lambda\mu)^*)=d(\mu^*\lambda^*).
 \end{align*}
 So suppose $w=\lambda^*\mu$. Let $m=d(\mu)\vee d(\lambda)$.  By Proposition~\ref{prop:lspfamily} we have
 \[
 q(\lambda^*\mu)=s_{\lambda^{*}}s_\mu=\sum_{\lambda\alpha=\mu\beta, 
\lambda\alpha \in \Lambda^{\leq m}} s_\alpha s_{\beta^{*}}.
 \]
In each summand, $\lambda\alpha=\mu\beta$ implies $d(w)=d(\mu)-d(\lambda)=d(\alpha)-d(\beta)$, so $q(w)\in \KP_R(\Lambda)_{d(w)}$ as needed. 

Now let $n\geq 2$ and suppose that $q(y)\in  \KP_R(\Lambda)_{d(y)}$ for every word $y$ with $|y|\leq n$. Let $w$ be a word with $|w|=n+1$ and $q(w)\not=0$. If $w$ contains a subword $w_iw_{i+1}=\lambda\mu$, then  $\lambda$ and $\mu$ are composable in $\Lambda$ since otherwise $q(\lambda\mu)=0$. Let $w'$ be the word obtained from $w$ by replacing $w_iw_{i+1}$ with the single path $\lambda\mu$. Then
\[
q(w)=s_{w_1}\cdots s_{w_{i-1}}s_{\lambda}s_{\mu}s_{w_{i+2}}\cdots s_{w_{n+1}}
=s_{w_1}\cdots s_{w_{i-1}}s_{\lambda\mu}s_{w_{i+2}}\cdots s_{w_{n+1}}=q(w').
\]
Since $|w'|=n$ and $d(w')=d(w)$, the inductive hypothesis implies that $q(w)\in \KP_R(\Lambda)_{d(w)}$. A similar argument shows that $q(w)\in \KP_R(\Lambda)_{d(w)}$ whenever $w$ contains a subword $w_iw_{i+1}=\lambda^*\mu^*$.

If $w$ contains no subword of the form $\lambda\mu$ or $\lambda^*\mu^*$, then, since $|w|\geq 3$, it must have a subword of the form $\lambda^*\mu$. By Proposition~\ref{prop:lspfamily} we write $q(w)$ as a sum of terms $q(y^i)$ with $|y^i|=n+1$ and $d(y^i)=d(w)$. Since $|w|\geq 3$, each nonzero summand $q(y^i)$ contains a factor of the form $s_{\beta^*}s_{\gamma^*}$ or one of the form $s_{\delta}s_\alpha$, and the argument above shows that every $q(y^i)\in \KP_R(\Lambda)_{d(w)}$. Thus  $q(w)\in \KP_R(\Lambda)_{d(w)}$ as well.
\end{proof}

\section{The uniqueness theorems}\label{sec-uniquenessthms}
Throughout this section, $\Lambda$ is a locally convex, row-finite $k$-graph, and $R$ is a commutative ring with $1$.

There are two uniqueness theorems in the theory of Kumjian-Pask algebras.  The graded-uniqueness theorem has no hypotheses on the graph, so 
applies very generally.  The Cuntz-Krieger uniqueness theorem assumes the graph satisfies an ``aperiodicity'' condition. 
The proofs of both theorems are straightforward once key helper-results (Lemma~\ref{lem:AaHRS4.3} and Proposition~\ref{prop:aperiodic}) have been established.

\begin{thm}[The graded-uniqueness theorem]\label{gut} Let $\Lambda$ be a locally convex, row-finite $k$-graph. 
  Suppose that $A$ is a $\Z^k$-graded $R$-algebra and 
$\phi:\KP_R(\Lambda) \to A$ is a graded $R$-algebra 
homomorphism. If $\phi(rp_v)\neq 0$ for all $r \in R \setminus \{0\}$ and $v \in \Lambda^0$, 
then $\phi$ is injective.
\end{thm}

\begin{thm}[The Cuntz-Krieger uniqueness theorem]\label{ckthm}
Let $\Lambda$ be a locally convex, row-finite $k$-graph satisfying the aperiodicity condition 
\begin{equation}\label{eq:aperiodic}
\text{for every $v \in \Lambda^0$, there exists $x\in v\Lambda^{\leq \infty}$
such that $\alpha \neq \beta$  implies $\alpha x \neq \beta x$.}
\end{equation}
Let $\phi:\KP_R(\Lambda) \to A$ 
be an $R$-algebra homomorphism into an $R$-algebra $A$. If $\phi(rp_v) \neq 0$ for all $r \in R \setminus \{0\}$ and 
$v \in \Lambda^0$, then $\phi$ is injective.
\end{thm}

The aperiodicity condition \eqref{eq:aperiodic} we have chosen to use is from  
\cite[Theorem~4.3]{RSY03}, is often called `condition $B$' in the literature, and generalises the many  notions of
 `aperiodicity' for $k$-graphs without sources.  See Lemma~\ref{lem-pain} below for more details.

We start by establishing that every nonzero element of $\KP_R(\Lambda)$ can be 
written in a certain form; this form differs from the one for graphs without sources only in the use of 
$\Lambda^{\leq n}$ in place of $\Lambda^n$.
    
\begin{lemma}
\label{lem:11}
Every nonzero $a \in \KP_R(\Lambda)$ 
can be written in normal form: that is, there exists $n \in \N^k \setminus\{0\}$ 
and a finite subset $F$ of $\Lambda \times \Lambda^{\leq n}$ such that
\[ a = \sum_{(\alpha, \beta) \in F} r_{\alpha, \beta}s_\alpha s_{\beta^{*}}\]
where $r_{\alpha,\beta} \in R\setminus\{0\}$ and $s(\alpha)=s(\beta)$ for all $(\alpha,\beta)\in F$.
\end{lemma}

\begin{proof}
Let $0\neq a \in \KP_R(\Lambda)$. By Corollary~\ref{cor-span} we can write $a$ as a finite sum 
 \[a=\sum_{(\mu,\nu) \in G} r_{\mu,\nu}s_\mu s_{\nu^{*}},\]
where $G \subseteq \Lambda \times \Lambda$, $s(\mu)=s(\nu)$ and each $r_{\mu,\nu}\neq 0$ in $R$. 
Let $n=\vee_{(\mu,\nu) \in G} d(\nu)$. Consider $(\mu,\nu)\in G$ such that $d(\nu)<n$. Then, using (KP2) and  (KP4$'$), we get
\[
s_\mu s_{\nu^{*}} = s_\mu p_{s(\mu)} s_{\nu^{*}}=s_\mu\Big(\sum_{\lambda \in s(\mu)\Lambda^{\leq n-d(\nu)}} s_\lambda s_{\lambda^{*}}\Big) s_{\nu{*}}=\sum_{\lambda \in s(\mu)\Lambda^{\leq n-d(\nu)}}s_{\mu\lambda} s_{(\nu\lambda)^{*}}.
\]
Substituting back into the expression for $a$ and combining terms gives the result.
\end{proof}

Next we generalise \cite[Lemma~2.3(1)]{BaH} to graphs with possible sources.

\begin{lemma}
\label{lem:AaHRS4.3}
Let $0\neq a=\sum_{(\alpha,\beta) \in F} r_{\alpha,\beta}s_\alpha s_{\beta^*}\in\KP_R(\Lambda)$ be in normal form. For all $(\mu,\nu)\in F$, 
\begin{equation}\label{eq:4.2}
0 \neq s_{\mu^*}as_\nu = r_{\mu,\nu}p_{s(\mu)}+\sum_{(\alpha,\nu)\in F, d(\alpha)\neq d(\mu)} 
r_{\alpha,\nu}s_{\mu^*}s_\alpha,\end{equation}
and the $0$-graded component $r_{\mu,\nu}p_{s(\mu)}$ of $s_{\mu^*}as_\nu$ is nonzero.
\end{lemma}

\begin{proof} Since $a$ is in normal form, $F\subseteq \Lambda\times \Lambda^{\leq n}$ for some $n \in \N^k \setminus\{0\}$. Fix $(\mu,\nu)\in F$. For all $(\alpha,\beta)\in F$, both $\beta$ and $\nu\in\Lambda^{\leq n}$, and hence (KP3$'$) gives $s_{\beta^*}s_\nu=\delta_{\beta,\nu}p_{s(\nu)}$. Thus
\begin{align*}
s_{\mu^*}as_\nu 
&=\sum_{(\alpha,\beta) \in F} r_{\alpha,\beta}s_{\mu^*}s_\alpha s_{\beta^*} s_\nu 
=\sum_{(\alpha,\nu) \in F} r_{\alpha,\nu}s_{\mu^*}s_\alpha p_{s(\nu)}\\
&= r_{\mu,\nu}p_{s(\mu)}+\sum_{(\alpha,\nu) \in F, \alpha\neq\mu } r_{\alpha,\nu}s_{\mu^*}s_\alpha
\end{align*}
by (KP2) since $(\alpha,\nu)\in F$ implies $s(\alpha)=s(\nu)$. If  there exists $\alpha$ such that $(\alpha,\nu)\in F$, $\alpha\neq\mu$ and $d(\alpha)=d(\mu)$, then (KP3$'$) implies that $s_{\mu^*}s_\alpha=0$. Thus 
\[
s_{\mu^*}as_\nu = r_{\mu,\nu}p_{s(\mu)}+\sum_{(\alpha,\nu)\in F, d(\alpha)\neq d(\mu)} r_{\alpha,\nu}s_{\mu^*}s_\alpha.
\]
It now follows that the $0$-graded component of $s_{\mu^*}as_\nu $ is $r_{\mu,\nu}p_{s(\mu)}$, which is nonzero by Theorem~\ref{thm-KP}\eqref{item-a}.  Now $s_{\mu^*}as_\nu$ is nonzero because its $0$-graded component is.
\end{proof}

\begin{proof}[Proof of Theorem~\ref{gut}]
  Let $0\neq a \in \KP_R(\Lambda)$.  Write
$a = \sum_{(\alpha, \beta) \in F} r_{\alpha, \beta}s_\alpha s_{\beta^{*}}$  in normal form. Let $(\mu,\nu)\in F$. By Lemma~\ref{lem:AaHRS4.3}, 
\[
0 \neq s_{\mu^*}as_\nu = r_{\mu,\nu}p_{s(\mu)}+\sum_{(\alpha,\nu)\in F, d(\alpha)\neq d(\mu)} 
r_{\alpha,\nu}s_{\mu^*}s_\alpha
\]
with $0$-graded component $r_{\mu,\nu}p_{s(\mu)}$.
Since $\phi$ is graded, 
$\phi(r_{\mu,\nu}p_{s(\mu)})$ is the $0$-graded component of 
$\phi(s_{\mu^*}as_{\nu})$.  Now $\phi(s_{\mu^*})\phi(a)\phi(s_{\nu})=\phi(s_{\mu^*}as_{\nu})\neq 0$ because its $0$-graded component $\phi(r_{\mu,\nu}p_{s(\mu})\neq 0$ by assumption. It follows that $\phi(a) \neq 0$ as well. Thus $\phi$ is injective.
\end{proof}

One immediate application of the Theorem~\ref{gut} is:

\begin{prop}
 Let $\Lambda$ be a locally convex, row-finite $k$-graph.  Then $\KP_{\C}(\Lambda)$ is isomorphic to 
a dense subalgebra of $C^*(\Lambda)$.
\end{prop}

\begin{proof}
 Let $(q,t)$ be a generating Cuntz-Krieger $\Lambda$-family in $C^*(\Lambda)$ as in \cite[Definition~3.3]{RSY03}.  Then $(q,t)$ is  a Kumjian-Pask $\Lambda$-family 
in $C^*(\Lambda)$. (To see this, recall that (KP1) and 
(KP2) hold, (KP4$'$) and (CK4) are the same and (KP3$'$) and (CK3) are the same via Corollary~\ref{cor-altKP3'}.)  Thus the universal property of 
$\KP_\C(\Lambda)$ of Theorem~\ref{thm-KP} gives  a homomorphism $\pi_{q,t}$ from $\KP_{\C}(\Lambda)$ onto the dense 
subalgebra \[A:=\lsp \{t_{\lambda}t_{\mu}^*:\lambda, \mu \in \Lambda\}\]
of $C^*(\Lambda)$.
To show that $\pi_{q,t}$ is injective, we will use the gauge-invariant uniqueness theorem, Theorem~\ref{gut}. 
Consider the subgroups \[
 A_n:= \lsp\{t_{\lambda}t_{\mu}^*:d(\lambda)- d(\mu)=n\}\quad(n\in\Z^k)\]
of $A$. A calculation using Corollary~\ref{cor-altKP3'} shows that $A_nA_m \subseteq A_{n+m}$. 
Since each spanning element $t_\mu t_{\nu}^*$ of $A$ belongs to $A_{d(\lambda)-d(\mu)}$, every element $a$ of $A$ can be written as a finite sum $\sum a_n$ with $a_n\in A_n$.  If $a_n\in A_n$ and  a finite sum $\sum a_n=0$, then each $a_n=0$ by the argument of the proof of \cite[Lemma~7.4]{ACaHR}, which uses the gauge action  of $\T^k$ on $C^*(\Lambda)$. Thus $\{A_n:n\in\Z^k\}$ is a grading of $A$.  It follows that 
$\pi_{q,t}$ is graded and hence is injective by Theorem~\ref{gut}.
\end{proof}

\begin{prop}
 \label{prop:aperiodic}
Let $\Lambda$ be a locally convex, row-finite  $k$-graph satisfying the aperiodicity condition \eqref{eq:aperiodic}.
Suppose $0\neq a=\sum_{(\alpha,\beta) \in F} 
r_{\alpha,\beta}s_\alpha s_{\beta^*}\in \KP_R(\Lambda)$ is in normal form. Let $(\mu, \nu) \in F$.
Then there exist $\sigma, \tau \in \Lambda$  such 
that $s_{\sigma^*}as_\tau=r_{\mu,\nu}p_{s(\mu)}$.
\end{prop}

\begin{proof}  Let $(\mu,\nu)\in F$. By Lemma~\ref{lem:AaHRS4.3}, 
\begin{equation}
\label{eq:Lem4.4}
0 \neq s_{\mu^*}as_\nu = r_{\mu,\nu}p_{s(\mu)}+\sum_{(\alpha,\nu)\in F, d(\alpha)\neq d(\mu)} 
r_{\alpha,\nu}s_{\mu^*}s_\alpha.
\end{equation}
If $G:=\{\alpha: (\alpha,\nu)\in F, d(\alpha)\neq d(\mu)\}=\emptyset$ then we can take $\sigma=\mu$ and $\tau=\nu$, and we are done. So suppose $G\neq \emptyset$.

Choose $y \in s(\mu)\Lambda^{\leq \infty}$ such that \eqref{eq:aperiodic} holds. Then for each
$\alpha \in G$, $\alpha y \neq \mu y$. 
So there exists $m_{\alpha}\in\N^k$ such that
$(\alpha y)(0,m_{\alpha}) \neq (\mu y)(0,m_{\alpha})$.  Let 
$m:=  \bigvee_{\alpha \in G} m_\alpha$.    For later use we note that since $m_\alpha \leq d(y) + d(\alpha)$, taking the meet of both sides with $m$ we get
\[m_\alpha=m \wedge m_{\alpha} \leq m \wedge (d(y) + d(\alpha)) \leq m \wedge d(y) + d(\alpha);
\]
the same argument gives $m_{\alpha} \leq m \wedge d(y) + d(\mu)$.

We would like to use $y(0, m)$ now, but since $y$ is a boundary path this may not be well-defined, so we will use $y(0, m\wedge d(y))$ instead.
Then, using (\ref{eq:Lem4.4}), 
\begin{align*}
s_{y(0,m\wedge d(y))^*}&(s_{\mu^*}as_\nu)s_{y(0,m\wedge d(y))}\\
&=r_{\mu,\nu}s_{y(0,m\wedge d(y))^*}p_{s(\mu)}s_{y(0,m\wedge d(y))} + 
\sum_{\alpha \in G}r_{\alpha,\nu}s_{y(0,m\wedge d(y))^*}s_{\mu^*}s_\alpha s_{y(0,m\wedge d(y))}\\
&=r_{\mu,\nu}p_{s(\mu)}+ \sum_{\{\alpha \in G: \alpha \neq \mu\}} 
r_{\alpha,\nu}s_{(\mu y(0,m\wedge d(y)))^*}s_{\alpha y(0,m\wedge d(y))} \\
\intertext{because $r(y)=s(\mu)$ and (KP2), and then (KP3), all on the first summand. After composing paths this}
&=r_{\mu,\nu}p_{s(\mu)}+ \sum_{\alpha \in G}
 r_{\alpha,\nu}s_{(\mu y)(0,m\wedge d(y)+d(\mu))^*}s_{(\alpha y)(0,m\wedge d(y)+d(\alpha))}.
\end{align*}
We claim that for every $\alpha \in G$, $s_{(\mu y)(0,m\wedge d(y)+d(\mu))^*}s_{(\alpha y)(0,m\wedge d(y)+d(\alpha))}=0$. 
To see this, by way of contradiction, suppose there
exists $\alpha \in G$ such that  
\[s_{(\mu y)(0,m\wedge d(y)+d(\mu))^*}s_{(\alpha y)(0,m\wedge d(y)+d(\alpha))} \neq 0.\] 
Since $m_\alpha\leq m\wedge d(y)+d(\mu)$ and $m_\alpha\leq m\wedge d(y)+d(\alpha)$  we must have  $s_{(\mu y)(0,m_\alpha)^*}s_{(\alpha y)(0,m_\alpha)} \neq 0$. But $(\mu y)(0,m_\alpha)$ and ($\alpha y)(0,m_\alpha)\in\Lambda^{\leq m_\alpha}$, and hence $(\mu y)(0,m_\alpha)=(\alpha y)(0, m_\alpha)$ by (KP3$'$), 
contradicting our choice of $m_{\alpha}$.  This proves the claim. 
Therefore
\[s_{y(0,m\wedge d(y))^*}(s_{\mu^*}as_\nu)s_{y(0,m\wedge d(y))}=r_{\mu,\nu}p_{s(\mu)}.\]
The proposition follows with $\sigma:=\mu y(0,m\wedge d(y))$ and  
$\tau:=\nu y(0,m\wedge d(y))$.
\end{proof}

\begin{proof}[Proof of Theorem~\ref{ckthm}]  Let $0\neq a \in \KP_R(\Lambda)$.  Write
$a = \sum_{(\alpha, \beta) \in F} r_{\alpha, \beta}s_\alpha s_{\beta^{*}}$  in normal form.
Fix $(\mu,\nu)\in F$. Since $\Lambda$ satisfies the aperiodicity condition \eqref{eq:aperiodic}, 
by Proposition~\ref{prop:aperiodic} 
there exist $\sigma, \tau \in \Lambda$ such that 
$s_\sigma^*as_\tau = r_{\mu, \nu}p_{s(\mu)}$. Now
$\phi(s_\sigma^*)\phi(a)\phi(s_\tau)=\phi(s_\sigma^*as_\tau)=\phi(r_{\mu, \nu}p_{s(\mu)}) \neq 0$ by assumption, and hence  $\phi(a) \neq 0$  as well.  Thus $\phi$ is injective.
\end{proof}



\section{Examples of $2$-graphs with sources\\ and applications to rank-$2$ Bratteli diagrams}\label{sec-Bratteli}

Throughout this section, $R$ is a commutative ring with $1$.
Let  $\Lambda$ be a $2$-graph.  We refer to the morphisms of degree $(n_1,0)$ as blue paths, 
to the morphisms of degree $(0,n_2)$ as red paths, and  write $\Lambda^\blue$ and $\Lambda^\red$ for the collection of blue and red paths, respectively.  
We say a path $\lambda \in \Lambda$ with $d(\lambda)\neq 0$ is a \emph{cycle} if
$r(\lambda)=s(\lambda)$ and for $0 <n\ < d(\lambda)$,  $\lambda(n) \neq s(\lambda)$.  
We say a cycle $\lambda$ is \emph{isolated}
if for every $0 < n\ \leq d(\lambda)$,  the sets 
\[
 r(\lambda)\Lambda^n \setminus \{\lambda(0,n)\} \text{ and }  s(\lambda)\Lambda^n \setminus \{\lambda(d(\lambda)-n,d(\lambda))\}
\]
are empty.

\begin{prop}\label{prop3.5}
Let $\Lambda$ be a finite $2$-graph such that
\begin{equation}\label{3.1}
\ \hfill
\parbox{0.9\textwidth}{
$\Lambda^\blue$ contains no cycles and each vertex $v \in \Lambda^0$
is the range of an isolated cycle in $\Lambda^\red$.
}\end{equation}
Let $S$ be the set of sources in $\Lambda^\blue$. Suppose the vertices in $S$ all lie on a single isolated cycle in $\Lambda^\red$. Let $Y$ denote the set $\Lambda^\blue S$ of blue paths with source in $S$.  
Then $\KP_R(\Lambda)$ is isomorphic to $M_Y(R[x,x^{-1}])$, where $R[x,x^{-1}]$ is the ring of Laurent polynomials over $R$.
\end{prop}

Any graph satisfying \eqref{3.1}  is locally convex \cite[page~141]{PRRS}.
Proposition~\ref{prop3.5} is very similar to Proposition~3.5 of \cite{PRRS}, which, with the same hypotheses on $\Lambda$ and $R=\C$, gives an isomorphism of $C^*(\Lambda)$ onto $M_Y(C(\T))\cong M_Y(\C)\otimes C(\T)$. The proof of \cite[Proposition~3.5]{PRRS} finds a family $\{\theta(\alpha,\beta):\alpha,\beta\in Y\}$ of matrix units  and a unitary $U$ in $C^*(\Lambda)$ such that $\theta(\alpha,\beta)U=U\theta(\alpha,\beta)$ for all $\alpha, \beta\in Y$.  This gives a homomorphism $\phi:M_Y(\C)\otimes C(\T)$ into $C^*(\Lambda)$.  The argument then shows that $U$ has full spectrum, from which it follows that $\phi$ is injective, and then $\phi$ is shown to be surjective.
Since $R$ is a ring with $1$, the  matrix units $\{\theta(\alpha,\beta):\alpha,\beta\in Y\}$ (and the unitary $U$) live naturally in $\KP_R(\Lambda)$. But in the algebraic setting, the spectral argument is not available; we  use the following lemma instead.
 
\begin{lemma}\label{lemma-rowan}
Let $W$ be a ring with $1$ where $W$ contains a set $\{\theta(\alpha,\beta):\alpha,\beta \in Y\}$ of matrix units. Let $\alpha_0\in Y$ and $D:=\theta(\alpha_0,\alpha_0) W\theta(\alpha_0,\alpha_0)$. Then the map $w\mapsto (\theta(\alpha_0,\alpha)w\theta(\beta,\alpha_0))_{\alpha,\beta}$ is an isomorphism of $W$ onto $M_Y(D)$.
\end{lemma}
The proof of Lemma~\ref{lemma-rowan} is straightforward, see, for example, \cite[Proposition~13.9]{rowan}. In the proof of Proposition~\ref{prop3.5} we will apply Lemma~\ref{lemma-rowan} with $\theta(\alpha_0,\alpha_0)=p_\dagger$,  where $\dagger$ is a vertex in the set $S$ of sources in $\Lambda^\blue$.

\begin{lemma}\label{lem-badapple}
Let $\Lambda$ be a $2$-graph satisfying the hypotheses of Proposition~\ref{prop3.5}. Fix a vertex $\dagger\in S$ and let $\mu$ be the unique red path of  least nonzero degree with range and source $\dagger$. Then there is an isomorphism of $R[x,x^{-1}]$ onto $p_{\dagger}\KP_R(\Lambda)p_{\dagger}$ such that $1\mapsto p_{\dagger}$, $x\mapsto s_\mu$ and $x^{-1}\mapsto s_{\mu^*}$.
\end{lemma}

\begin{proof} Let $\alpha\in\Lambda$ with $r(\alpha)=\dagger$. Then $\alpha\in\Lambda^\red$. Let $\nu$ be the unique red path of nonzero least length $n_2$ such that $s(\alpha\nu)=\dagger$.  For any $\beta$, applying (KP4$'$) to  $s_\alpha s_{\beta^*}$ gives 
\begin{equation}\label{mu}
s_\alpha s_{\beta^*}=s_\alpha p_{s(\alpha)}s_{\beta^*}=s_\alpha\sum_{\lambda\in s(\alpha)\Lambda^{\leq(0,n_2)}}s_\lambda s_{\lambda^*}s_{\beta^*}=s_{\alpha\nu}s_{(\beta\nu)^*}.
\end{equation}
The Kumjian-Pask relation (KP4$'$) at $\dagger$ with $n=|\mu|e_2$ says that $s_\mu s_{\mu^*}=p_{\dagger}$.
Writing $\mu^i$ for the path which traverses $\mu$ exactly $i$ times, we now have 
\begin{align*}
p_{\dagger}\KP_R(\Lambda)p_{\dagger}&=\lsp\{s_\alpha s_{\beta^*}: \alpha,\beta\in\Lambda^\red, s(\alpha)=s(\beta), r(\alpha)=r(\beta)=\dagger\}\\
&=\lsp\{s_{\mu^i} s_{(\mu^j)^*}:i,j\in\N\} \quad\text{(using \eqref{mu})}\\
&=\lsp\{s_{\mu^i}, s_{(\mu^i)^*}:i\in\N\}\quad\text{(since $s_\mu s_{\mu^*}=p_\dagger$)}.
\end{align*}

Now let $E$ be the directed graph with one vertex $w$ and one edge $f$. Let $L_R(E)$ be the Leavitt path algebra over $R$, and let $\{q_w, t_f, t_{f^*}\}$ be the generating Leavitt $E$-family in $L_R(E)$. The  polynomials  $\{1, x, x^{-1}\}$ are a Leavitt $E$-family in $R[x,x^{-1}]$, and the universal property of $L_R(E)$ gives an $R$-algebra homomorphism $\rho:L_R(E)\to R[x,x^{-1}]$ such that $\rho(1)=q_w$, $\rho(t_f)=x$ and $\rho(t_{f^*})=x^{-1}$. 
Now $\rho$ is surjective because the range of $\rho$ contains the generators $x$ and $x^{-1}$ of $R[x,x^{-1}]$, and it is one-to-one by definition of the zero polynomial.

Since $p_{\dagger}=s_\mu s_{\mu^*}$, it follows that $\{p_\dagger, s_\mu, s_{\mu^*}\}$ is a Leavitt $E$-family in $\KP_R(\Lambda)$. By the universal property of $L_R(E)$ again, there is an $R$-algebra homomorphism $\pi:L_R(E)\to \KP_R(E)$ such that $\pi(q_w)=p_\dagger$, $\pi(t_f)=s_\mu$ and $\pi(t_{f^*})=s_{\mu^*}$.  Now we observe that $p_{\dagger}\KP_R(\Lambda)p_{\dagger}=\lsp\{s_{\mu^i}, s_{(\mu^i)^*}:i\in\N\}$ is graded over $\Z$ by the subgroups $\lsp\{s_{\mu^i}\}$. Thus $\pi$ is graded, and hence it is injective by the graded-uniqueness theorem \cite[Theorem~5.3]{T}.  

Now $\pi\circ\rho^{-1}:R[x,x^{-1}]\to \KP_R(\Lambda)$ is injective and satisfies 
\[\pi\circ\rho^{-1}(1)=p_\dagger, \pi\circ\rho^{-1}(x)=s_\mu \text{ and } \pi\circ\rho^{-1}(x^{-1})=s_{\mu^*}.\] It follows that $\pi\circ\rho^{-1}$ has range $\lsp\{s_{\mu^i}, s_{(\mu^i)^*}:i\in\N\}=p_{\dagger}\KP_R(\Lambda)p_{\dagger}$. Thus $\pi\circ\rho^{-1}$ is the required isomorphism. 
\end{proof}

\begin{proof}[Proof of Proposition~\ref{prop3.5}]  Fix a vertex $\dagger\in S$ and 
let $e_\dagger$ be the edge in $S\Lambda^\red S$ with $r(e_\dagger)=\dagger$.  For $\alpha,\beta$ in $Y$, let $\nu(\alpha,\beta)$ be the unique path in $\Lambda^\red$ connecting $s(\alpha)$ and $s(\beta)$ such that $\nu(\alpha,\beta)$ does not contain $e_\dagger$. 
It follows from the proof of Lemma~3.8 of \cite{PRRS}, that the elements
 \begin{equation}\label{matrixunits}
 \theta(\alpha,\beta):=\begin{cases}s_\alpha s_{\nu(\alpha,\beta)}s_{\beta^*}  
 &\text{if $s(\nu(\alpha,\beta))=s(\beta)$};\\
 s_\alpha s_{\nu(\alpha,\beta)^*}s_{\beta^*}  &\text{if $s(\nu(\alpha,\beta))=s(\alpha)$}
 \end{cases}
 \end{equation}
form a set $\{\theta(\alpha,\beta):\alpha,\beta\in Y\}$ of matrix units in $\KP_R(\Lambda)$, that is, $\theta(\alpha,\beta)\theta(\sigma,\tau)=\delta_{\beta, \sigma}\theta(\alpha,\tau)$. Since $\Lambda^0$ is finite, there exists $n_1$ such that $Y=\Lambda^{\leq (n_1,0)}$. Calculate:
\begin{align*}\sum_{\alpha\in Y} \theta(\alpha,\alpha)&=\sum_{\alpha\in Y} s_\alpha s_{\alpha^*}=\sum_{v\in\Lambda^0}\sum_{\alpha\in vY}s_\alpha s_{\alpha^*}
=\sum_{v\in\Lambda^0}\sum_{\alpha\in v\Lambda^{\leq (n_1,0)}}s_\alpha s_{\alpha^*}=\sum_{v\in\Lambda^0}p_v=1.
\end{align*}

Let $\alpha_0=\dagger$. Then $\alpha_0$ is a path of degree $0$ in $Y$, and $\theta(\alpha_0,\alpha_0)=p_\dagger$.
Now apply Lemma~\ref{lemma-rowan} with $\alpha_0=\dagger$  to see that $\KP_R(\Lambda)$ is isomorphic to $M_Y(p_\dagger\KP_R(\Lambda)p_\dagger)$.  The proposition now follows from the isomorphism of $p_\dagger\KP_R(\Lambda)p_\dagger$ with $R[x,x^{-1}]$ of Lemma~\ref{lem-badapple}
\end{proof}

\subsection{Application to rank-$2$ Bratteli diagrams} \label{sec_cofinal}  Throughout this subsection, 
$\Lambda$ is a row-finite $2$-graph without sources which is a rank-$2$ Bratteli diagram in the sense of \cite[Definition~4.1]{PRRS}. 
This means that the blue subgraph $\Lambda^\blue$ of paths of degree $(n_1, 0)$ is a Bratteli diagram in the usual sense: the vertex set $\Lambda^0$ is the disjoint union $\bigsqcup_{n=0}^\infty V_n$ of finite subsets $V_n$, each blue edge goes from some $V_{n+1}$ to $V_n$. The red subgraph $\Lambda^\red$ of paths of degree $(0,n_2)$ consists of disjoint cycles whose vertices lie entirely in some $V_n$. For each blue edge $e$ there is a unique red edge $f$ with $s(f)=r(e)$, and hence by the factorization property there is a unique blue-red path $\Ff(e)h$ such that $\Ff(e)h=fe$. The map $\Ff:\Lambda^{e_1}\to \Lambda^{e_1}$ is a bijection, and induces a permutation of each finite set $\Lambda^{e_1}V_n$. We write $o(e)$ for the order of $e$: the smallest $l>0$ such that $\Ff^l(e)=e$. 

\label{page-cofinal} A $k$-graph without sources is \emph{cofinal} if for every $x\in\Lambda^{\infty}$ and every $v\in \Lambda^0$, there exists $n\in\N^k$ such that $v\Lambda x(n)\neq \emptyset$.
It follows from \cite[Theorem~5.1]{PRRS} that if $\Lambda$ is cofinal and $\{o(e)\}$ is unbounded, then $\Lambda$ is aperiodic; and hence if $K$ is a field, then $\KP_K(\Lambda)$ is simple  by \cite[Corollary~7.8]{ACaHR}.  So the cofinal rank-$2$ Bratteli diagrams give a rich supply of simple Kumjian-Pask algebras.

Theorem~7.10 of \cite{ACaHR} uses rank-$2$ Bratteli diagrams to show that there are simple Kumjian-Pask algebras that are neither purely infinite nor locally matricial.  But the  analysis in \S7 of \cite{ACaHR} specialises to $K=\C$ because  the proof reduces to  a ``rank-$2$ Bratteli diagrams $\Lambda_N$ of depth $N$'', which is a $2$-graph with sources. In the absence of a Kumjian-Pask algebra for such graphs, the embedding into $C^*(\Lambda_N)$ is used.  This was one of our motivations for defining a Kumjian-Pask algebra for graphs with sources in the first place, and we now show that we can extend  
\cite[Theorem~7.10]{ACaHR} from $\C$ to arbitrary fields $K$ (see Theorem~\ref{thm-point} below).  Note that
we will use the letter $K$ in place of $R$ when we are explicitly assuming the underlying ring is a field.  
We start with three lemmas.  

\begin{lemma}\label{lem-helper-cut'} Let $\Lambda$ be a rank-$2$ Bratteli diagram.
Let $\Lambda_N$ be the rank-$2$ Bratteli diagram of depth $N$ consisting of all the paths in $\Lambda$ which begin and end in $\bigcup_{n=0}^N V_n$. Then $\Lambda_N$ is locally convex and  the subalgebra $\lsp\{s_\mu s_{\nu^*}:\mu,\nu\in\Lambda_N\}$ of $\KP_R(\Lambda)$ is canonically isomorphic to $\KP_R(\Lambda_N)$.
\end{lemma}

\begin{proof}  The graph  $\Lambda_N$ has sources, but satisfies \eqref{3.1}, and hence is locally convex \cite[page 141]{PRRS}.
For $v\in\Lambda_N^0$ and $\lambda\in\Lambda_N$, set $Q_v=p_v$, $T_\lambda=s_\lambda$ and $T_{\lambda^*}=s_{\lambda^*}$. We will show that $(Q,T)$ is a Kumjian-Pask $\Lambda_N$-family in $\KP_R(\Lambda)$. Both (KP1) and (KP2) for $(Q,T)$ immediately reduce to (KP1) and (KP2) for the $\Lambda$-family $(p,s)$.

For (KP3$'$), let $n\in\N^2\setminus\{0\}$ and $\lambda,\mu\in\Lambda_N^{\leq n}$. 
Since \[
T_{\lambda^*}T_\mu=T_{\lambda^*}Q_{r(\lambda)}Q_{r(\mu)}T_\mu=\delta_{r(\lambda), r(\mu)}T_{\lambda^*}T_\mu,\] if  $r(\lambda)\neq r(\mu)$, then $T_{\lambda^*}T_\mu=0=\delta_{\lambda,\mu}Q_{s(\mu)}$. So we may assume that $r(\lambda)=r(\mu)$. 

Notice that $d(\lambda)_2 = d(\mu)_2 = n_2$ because $\lambda, \mu \in \Lambda^{\leq n}$ and the red 
subgraph consists entirely of cycles and hence has no sources.  We claim that $d(\lambda)_1=d(\mu)_1$.  
By way of contradiction, 
suppose $d(\lambda)_1 < d(\mu)_1 \leq n_1$.  Since $r(\lambda)=r(\mu)$, we must have $s(\lambda)\in V_i$  and  $s(\mu)\in V_j$ where $i<j$.  
Thus $s(\lambda)\notin V_N$. But  $s(\lambda)$ 
must be  a source in the blue graph because $d(\lambda)<n$. This implies $s(\lambda)\in V_N$, a contradiction.
 Similarly, we cannot have $d(\mu)_1 < d(\lambda)_1$. Thus $d(\lambda)=d(\mu)$ as claimed. 

  Now  $\lambda,\mu\in r(\lambda)\Lambda^{d(\lambda)}$, and $T_{\lambda^*}T_\mu=s_{\lambda^*}s_\mu=\delta_{\lambda,\mu}p_{s(\mu)}=\delta_{\lambda,\mu}Q_{s(\mu)}$ by (KP3) for $(p,s)$.   Thus (KP3$'$) holds for $(Q,T)$.

To see  (KP4$'$) holds, we use Lemma~\ref{lem-kp4alt}. Let $v\in\Lambda^0_N$ and suppose that $v\Lambda_N^{e_i}\neq\emptyset$. If $v\in\Lambda_N^0\setminus V_N$, then $i=1$ or $2$, and  if $v\in V_N$, then $i=2$. In both cases,  (KP4) for $(q,t)$ gives \[Q_v=p_v=\sum_{\lambda\in v\Lambda^{e_i}}s_\lambda s_{\lambda^*}= \sum_{\lambda\in v\Lambda_N^{e_i}}s_\lambda s_{\lambda^*}.\]
Thus (KP4$'$) holds by Lemma~\ref{lem-kp4alt}.

Denote the generating Kumjian-Pask $\Lambda_N$-family in $\KP_R(\Lambda_N)$ by $(q, t)$.  The universal property of $\KP_R(\Lambda_N)$ (Theorem~\ref{thm-KP}(\ref{item-a})) now gives a homomorphism $\pi_{Q,T}:\KP_R(\Lambda_N)\to \KP_R(\Lambda)$ such that $\pi_{Q,T}(q_v)=p_v$, $\pi_{Q,T}(t_\lambda)=s_\lambda$ and $\pi_{Q,T}(t_{\lambda^*})=s_{\lambda^*}$.  Thus $\pi_{Q,T}$ has range $\lsp\{s_\mu s_{\nu^*}:\mu,\nu\in\Lambda_N\}$.
It is clear that $\pi_{Q,T}$ is graded. Since each $\pi_{Q,T}(q_v)\neq 0$, it follows from the graded-uniqueness theorem (Theorem~\ref{gut}) that $\pi_{Q,T}$ is injective.
\end{proof}

\begin{lemma}\label{break-up}  Let $\Lambda_N$ be a rank-$2$ Bratteli diagram of depth $N$.
Let $V_{N,i}$ be the vertices in $V_N$ of one isolated red cycle, and let $\Lambda_{N,i}$ be the subgraph of $\Lambda_N$ of paths with source and range in $V:=\{v\in \Lambda^{0}: v\Lambda V_{N,i}\neq \emptyset\}$. Let $(q,t)$ be a generating Kumjian-Pask $\Lambda_N$ family in $\KP_R(\Lambda_N)$. Then the subalgebra \[C_{n,i}:=\lsp\{t_\mu t_{\nu^*}:\mu,\nu\in\Lambda_N, s(\mu)=s(\nu)\in V_{N,i}\}\] of $\KP_R(\Lambda_N)$ is canonically isomorphic to $\KP_R(\Lambda_{N,i})$.
\end{lemma}

\begin{proof} Set $P_\cycle=\sum_{\alpha\in\Lambda^{\blue}V_{N,i}} t_\alpha t_{\alpha^*}$.   
Let $v\in \Lambda_{N,i}^0$ and $\mu \in  \Lambda_{N,i}$. 
Set \[Q_v=P_iq_v P_i, T_\mu=P_i t_\mu P_i \text {and } T_{\mu^*}=P_i t_{\mu^*} P_i.\] 
We will show  that $(Q,T)$ is a Kumjian-Pask $\Lambda_{N,i}$-family in $C_{N,i}$. 

If $\alpha,\beta\in v\Lambda^\blue V_{N,i}$, then $d(\alpha)=d(\beta)$ because $\Lambda^\blue$ is a Bratteli diagram.
Thus
\begin{align*}
P_i^2&=\sum_{v,w \in V}\sum_{\alpha\in v\Lambda^\blue V_{N,i}}\sum_{\beta\in w\Lambda^\blue V_{N,i}}t_\alpha t_{\alpha^*}q_v q_w t_\beta t_{\beta^*}
=\sum_{v \in V}\sum_{\alpha,\beta\in v\Lambda^\blue V_{N,i}}t_\alpha t_{\alpha^*}t_\beta t_{\beta^*}\\
&=\sum_{v \in V}\sum_{\alpha\in v\Lambda^\blue V_{N,i}}t_\alpha t_{\alpha^*}t_\alpha t_{\alpha^*}=P_i.
\end{align*}
We also have $q_vP_i=\sum_{\alpha\in v\Lambda^\blue V_{N,i}}t_\alpha t_\alpha^*=P_i q_v$, and 
\begin{align*}
P_it_\mu P_i&=\sum_{\alpha,\beta\in\Lambda^{\blue}V_{N,i}} t_\alpha t_{\alpha^*}t_\mu t_\beta t_{\beta^*}
=\sum_{\alpha\in r(\mu)\Lambda^{\blue}V_{N,i}}\sum_{\beta\in s(\mu)\Lambda^{\blue}V_{N,i}} t_\alpha t_{\alpha^*}t_{\mu \beta} t_{\beta^*}\\
&=\sum_{\beta\in s(\mu)\Lambda^{\blue}V_{N,i}} t_{\mu \beta} t_{\beta^*}
\quad\text{(by KP3$'$ because such $\alpha$ must have $d(\alpha)=d(\mu\beta)$)}\\
&=t_\mu P_i.
\end{align*} 
A similar calculation gives $P_it_{\mu^*} P_i=P_it_{\mu^*}$. The above relations help to reduce the Kumjian-Pask relations for $(Q,T)$ to the Kumjian-Pask relations for $(q,t)$ in $\KP_R(\Lambda_N)$.

To see  (KP1) holds, let $v,w\in \Lambda_{N,i}^0$. Then 
\[
Q_vQ_w=P_iq_vP_i^2q_wP_i=P_i^3q_vq_wP_i=\delta_{v,w}P_iq_vP_i=\delta_{v,w}Q_v.
\]
To see  (KP2) holds, let $\lambda,\mu\in\Lambda_{N,i}$.  Then
\[
T_\lambda T_\mu=P_it_\lambda P_i^2 t_\mu P_i=P_it_\lambda t_\mu P_i=P_i t_{\lambda\mu}P_i=T_{\lambda\mu},
\]
and, similarly, $T_{\mu^*} T_{\lambda^*}=T_{(\lambda\mu)^*}$.
To see (KP3$'$) holds, let $n\in\N^2\setminus\{0\}$ and $\lambda,\mu\in\Lambda_{N,i}^{\leq n}$. Then
\begin{align*}
T_{\lambda^*}T_\mu&=P_i t_{\lambda^*}P_it_\mu P_i=P_i t_{\lambda^*}t_\mu P_i\\
&=\delta_{\lambda,\mu}P_i q_{s(\mu)} P_i \quad\text{(because $\lambda,\mu\in\Lambda_N^{\leq n}$)}\\
&=\delta_{\lambda,\mu}Q_{s(\mu)}.
\end{align*}
To see (KP4$'$) holds, let $n\in\N^2\setminus\{0\}$ and $v\in\Lambda_{N,i}^0$. Then
\begin{align*}
\sum_{\lambda\in v\Lambda_{N,i}^{\leq n}} T_\lambda T_{\lambda^*}
&=\sum_{\lambda\in v\Lambda_{N,i}^{\leq n}} P_it_\lambda P_i^2 t_{\lambda^*}P_i=\sum_{\lambda\in v\Lambda_{N,i}^{\leq n}} P_it_\lambda  t_{\lambda^*}P_i\\
&=P_i\Big( \sum_{\lambda\in v\Lambda_{N,i}^{\leq n}} t_\lambda  t_{\lambda^*} \Big)P_i
=P_i\Big( \sum_{\beta\in v\Lambda_{N,i}^{\leq (n_1+N, n_2)}} t_\beta\ t_{\beta^*} \Big)P_i\\
&=P_iq_vP_i=Q_v.
\end{align*}
Thus $(Q,T)$ is a Kumjian-Pask $\Lambda_{N,i}$-family in $C_{N,i}$. Write $(q^i, t^i)$ for the universal Kumjian-Pask family in $\KP_R(\Lambda_{N,i})$. The  universal property of $\KP_R(\Lambda_{N,i})$ gives a homomorphism $\pi_{Q,T}: \KP_R(\Lambda_{N,i})\to C_{N,i}$ such that $\pi_{Q,T}(q_v^i)=Q_v$, $\pi_{Q,T}(t_\lambda^i)=T_\lambda$ and $\pi_{Q,T}(t_{\lambda^*}^i)=T_{\lambda^*}$. Thus $\pi_{Q,T}(t_\mu^i t_{\nu^*}^i)=P_it_\mu P_i t_{\nu^*}P_i=P_it_\mu t_{\nu^*}P_i$. Since $P_i$ is homogeneous of degree $0$ it follows that $\pi_{Q,T}$ is graded.  Thus $\pi_{Q,T}$ is injective by the graded-uniqueness (Theorem~\ref{gut}). Finally, the range of $\pi_{Q,T}$ is 
\begin{align*}\lsp\{P_i t_\mu P_i t_{\nu^*}P_i:\lambda,\mu\in \Lambda_N\}&=\lsp\{P_i t_\mu P_i t_{\nu^*}P_i: s(\mu)=s(\nu)\in V_{N,i}, r(\mu), r(\nu)\in V\}, \end{align*}
which is equal to $C_{N,i}$.
\end{proof}

\begin{lemma}\label{lem-restrict-iso}
 Let $\Lambda_N$ be a rank-$2$ Bratteli diagram of depth $N$.  Suppose that the set of sources in $\Lambda^\blue$ are the vertices on a single isolated cycle in $\Lambda^\red$.  Let $P$ be the idempotent $P=\sum_{v\in V_0}p_v$ and $X=V_0\Lambda^\blue V_N$. Then $P\KP_R(\Lambda_N)P$ is isomorphic to $M_X(R[x,x^{-1}])$.
\end{lemma}

\begin{proof}  Let $Y=\Lambda^\blue V_N$ and fix $\dagger\in V_N$. Let $\theta(\alpha, \beta)$ be the matrix units defined at \eqref{matrixunits}. The graph $\Lambda_N$ satisfies the hypotheses of Proposition~\ref{prop3.5}, and hence \[w\mapsto \big( \theta(\dagger, \alpha)w\theta(\beta,\dagger )\big)_{\alpha,\beta\in Y}\] is an isomorphism $\psi$ of $\KP_R(\Lambda_N)$ onto $M_Y(p_\dagger\KP_R(\Lambda_N)p_\dagger)$, and $p_\dagger\KP_R(\Lambda_N)p_\dagger$ is isomorphic to $R[x,x^{-1}]$ (see Lemmas~\ref{lemma-rowan} and \ref{lem-badapple} for how the isomorphism of Proposition~\ref{prop3.5} decomposes).  It suffices to prove  that the restriction of $\psi$ to $P\KP_R(\Lambda_N)P$ has range $M_X(p_\dagger\KP_R(\Lambda_N)p_\dagger)$.
We observe that \[\theta(\dagger,\alpha)PwP\theta(\beta,\dagger)=\begin{cases} \theta(\dagger,\alpha)w\theta(\beta,\dagger)&\text{if $r(\alpha), r(\beta)\in V_0$;}\\0&\text{else}. \end{cases}\] 
It follows first, that $\theta(\dagger,\alpha)PwP\theta(\beta,\dagger)\neq 0$ implies $\alpha,\beta\in X$, and second, if $\alpha,\beta\in X$, then \[\theta(\dagger,\alpha)P\KP_R(\Lambda_N)P\theta(\beta,\dagger)=\theta(\dagger,\alpha)\KP_R(\Lambda_N)\theta(\beta,\dagger)=p_\dagger\KP_R(\Lambda_N)p_\dagger.\]
Thus $\psi|$ is onto $M_X(p_\dagger\KP_R(\Lambda_N)p_\dagger)$.
\end{proof} 

An idempotent $p$ in $R$ is \emph{infinite} if there exist orthogonal nonzero idempotents $p_1, p_2 \in R$
and elements $x,y \in R$ such that
\[p=p_1+p_2,  \ x \in pRp_1,\   y \in p_1Rp,\ p=xy \text{ and } \ p_1=yx.\]  A simple ring is
\emph{purely infinite} if every nonzero right ideal of $R$ contains an infinite idempotent \cite[\S1]{AGP}.

\begin{prop}\label{prop-notpi}
Let $\Lambda$ be a rank-$2$ Bratteli diagram, and  $K$ a field. If $\Lambda$ is cofinal and aperiodic, then $\KP_K(\Lambda)$ is not purely infinite. 
\end{prop}

\begin{proof} Since $\Lambda$ is cofinal and aperiodic, and $K$ is a field, $\KP_K(\Lambda)$ is simple by \cite[Theorem~6.1]{ACaHR}.
Let $P_0:=\sum_{v\in V_0}p_v$.  Since the property of being purely infinite and simple 
passes to corners by \cite[Proposition~10]{AA2}, it suffices  to prove that $P_0\KP_K(\Lambda)P_0$ is not purely infinite. We argue by contradiction: suppose that $P_0\KP_K(\Lambda)P_0$ is purely infinite.
Then $P_0\KP_K(\Lambda)P_0$ contains an infinite idempotent $p$. Then there exist nonzero idempotents  $p_1$, $p_2$ and elements $x$, $y$ in $P_0\KP_K(\Lambda)P_0$ such that
\begin{equation}\label{infinrels}
p=p_1+p_2,\quad p_1p_2=p_2p_1=0,\quad xy=p\quad \text{and}\quad yx=p_1.
\end{equation}
Choose $N\in \N$ large enough to ensure that all five elements can be written as linear combinations of elements $s_\lambda s_{\mu^*}$ for which $s(\lambda)$ and $s(\mu)$ are in $\bigcup_{n=0}^N V_n$. 

Let $\Lambda_N$ be the rank-$2$ Bratteli diagram of depth $N$. Let $(q,t)$ be the generating Kumjian-Pask $\Lambda_N$-family in $\KP_K(\Lambda_N)$ and set $Q_0=\sum_{v\in V_0}q_v$.  Since $\KP_K(\Lambda_N)$ is canonically isomorphic to the subalgebra $\lsp\{s_\mu s_{\nu^*}:\mu,\nu\in\Lambda_N\}$ of $\KP_K(\Lambda)$ by Lemma~\ref{lem-helper-cut'}, we may assume that $p$, $p_1$, $p_2$, $x$ and $y$ are all in $Q_0\KP_K(\Lambda_N)Q_0$.

Next we will decompose $\KP_K(\Lambda_N)$ into a direct sum. Using (KP4$'$), \[\KP_K(\Lambda_N)=\lsp\{t_\lambda t_{\mu^*}:s(\lambda)=s(\mu)\in V_N\}.\] 
Partition $V_N$ into subsets $V_{N,i}$, each consisting of the vertices of one isolated red cycle. We claim that if   $s(\mu)$ and $s(\alpha)$ are in different $V_{N,i}$, then $t_{\mu^*}t_\alpha=0$. By way of contradiction, suppose not. By Proposition~\ref{prop:lspfamily}, there exists $(\sigma,\tau)\in\Lambda_N\times\Lambda_N$ such that $\mu\sigma=\alpha\tau$ and $d(\mu\sigma)=d(\mu)\wedge d(\alpha)$.  Since $s(\mu)=s(\alpha)\in V_N$, both $\sigma$ and $\tau$ are red paths.  But $s(\sigma)=s(\tau)$ implies $\sigma$ and $\tau$ lie on the same isolated red cycle, and this implies that $s(\mu)=r(\sigma)$ and $s(\alpha)=r(\tau)$ are in the same $V_{N,i}$, a contradiction.   Thus $t_{\mu^*}t_\alpha=0$, as claimed.  It follows that $\KP_K(\Lambda_N)$ is the direct sum of the subalgebras
\[
C_{N,i}:=\lsp\{t_\lambda t_{\mu^*}:s(\lambda)=s(\mu)\in V_{N,i}\},
\]
and we have
\[Q_0\KP_K(\Lambda_N)Q_0=\bigoplus_i Q_0C_{N,i}Q_0.\] 
The elements $p$, $p_1$, $p_2$, $x$ and $y$ in $Q_0\KP_K(\Lambda_N)Q_0$ all have direct sum decompositions, and the elements satisfy the relations \eqref{infinrels}. The component of $p_2$ is nonzero in at least one summand, and then the same component of the rest must be nonzero too. So we may assume that $p$, $p_1$, $p_2$, $x$ and $y$  are all in $Q_0C_{N,i}Q_0$ for some $i$.

Now consider the subgraph $\Lambda_{N,i}$ of $\Lambda_N$ of paths with source and range in $\{v\in \Lambda^{0}: v\Lambda V_{N,\cycle}\neq \emptyset\}$. By Lemma~\ref{break-up}, $C_{N,i}$ is canonically isomorphic to $\KP_K(\Lambda_{N,i})$, and by Proposition~\ref{prop3.5}, $\KP_K(\Lambda_{N,i})$ is isomorphic to a matrix algebra $M_Y(K[x,x^{-1}])$ for a certain set $Y$, and by Lemma~\ref{lem-restrict-iso} this isomorphism restricts to an isomorphism of  $Q_0C_{N,i}Q_0$ onto $M_X(K[x,x^{-1}])$ where $X\subseteq Y$.
 Pulling the five elements through all these isomorphisms gives us nonzero idempotents $q$, $q_1$, $q_2$ and elements $f$, $g$ in $M_X(K[x,x^{-1}])$ such that
\[
q=q_1+q_2,\quad q_1q_2=q_2q_1=0,\quad fg=q\quad \text{and}\quad gf=q_1.
\]

Evaluation at $z\in K$ is a homomorphism, and so  $f(z)g(z)=q(z)$ and $g(z)f(z)=q_1(z)$.  Thus $g(z)$ is an isomorphism of $q(z)K^X$ onto $q_1(z)K^X$. So the matrices $q(z)$ and $q_1(z)$ have the same rank. On the other hand, since $q_1(z)$ and $q_2(z)$ are orthogonal, $\rank(q_1(z)+q_2(z))=\rank q_1(z)+\rank q_2(z)$. Now  $q=q_1+q_2$ implies that $\rank q_2(z)=0$ for all $z$. This contradicts that $q_2$ is nonzero. Thus there is no infinite idempotent  in $P_0\KP_K(\Lambda)P_0$, as claimed. Thus $P_0\KP_K(\Lambda)P_0$ is not purely infinite, and neither is $\KP_K(\Lambda)$. 
\end{proof}

Recall that a matricial algebra is a finite direct product of full matrix algebras and
we say an algebra is locally matricial algebra if it is direct limit of matricial algebras.

\begin{thm}\label{thm-point} Let $\Lambda$ be a rank-$2$ Bratteli diagram which is cofinal and aperiodic, and $K$ a field. Then $\KP_K(\Lambda)$ is simple, but is neither purely infinite nor locally matricial.
\end{thm}

\begin{proof}
By \cite[Theorem~6.1]{ACaHR},  $\KP_K(\Lambda)$ is simple, and by Proposition~\ref{prop-notpi} it is not purely infinite. To see that it is not locally matricial, consider the element $s_\mu$ associated to a single red cycle $\mu$. Since $v:=r(\mu)=s(\mu)$ receives just one red path of length $|\mu|$, namely $\mu$, the Kumjian-Pask relation (KP4$'$) at $v$  for $n=|\mu|e_2$ (which only involves red paths) says that $p_v=s_\mu s_\mu^*$. But, as seen in the proof of Lemma~\ref{lem-badapple}, the subalgebra generated by  $s_\mu$ is isomorphic to the Leavitt path algebra of the directed graph consisting of a single vertex $w$ and a single loop $e$ at $w$, which is in turn isomorphic to $K[x,x^{-1}]$. Thus $s_\mu$ generates an infinite-dimensional algebra, and  does not lie in a finite-dimensional subalgebra.
\end{proof}


\section{Desourcification}
Throughout this section,  $\Lambda$ is a locally convex, row-finite $k$-graph.
In Theorem~\ref{thm:m_context}  we show that the Kumjian-Pask algebra of $\Lambda$ is
Morita equivalent to a Kumjian-Pask algebra of a certain $k$-graph $\tL$ 
without sources. The graph $\tL$  is called the \emph{desourcification of $\Lambda$}.  The construction of $\tL$, and the Morita equivalence of the $C^*$-algebras $C^*(\Lambda)$ and $C^*(\tL)$ goes back to Farthing \cite{F}.  Farthing's construction was refined by Robertson and Sims in \cite{RS2009}, and then generalised to finitely aligned $k$-graphs by Webster in \cite{W}. The difference in the approaches is that Farthing's construction adds paths and vertices to $\Lambda$ to obtain a graph $\overline{\Lambda}$, and the Robertson-Sims-Webster construction abstractly builds a $\tilde\Lambda$ which is then shown to contain a copy of $\Lambda$. When $\Lambda$ is locally convex, 
$\overline\Lambda$ and $\tilde\Lambda$ are isomorphic.

We use $\tilde\Lambda$, and  start by giving the details about $\tilde\Lambda$ that we need. Set \[V_\Lambda:=\{(x;m):x \in \Lambda^{\leq \infty}, m \in \N^k\}\text{\ and\ }P_\Lambda:=\{(x;(m,n)):
x \in \Lambda^{\leq \infty}, m \leq n \in \N^k\};\] 
the vertices and paths of $\tilde\Lambda$ are quotients of these sets, respectively.  
For this, define relations $\approx$ on $V_\Lambda$ and $\sim$ on 
$P_\Lambda$ as in Definitions~4.2 and~4.3 of \cite{W}. First, 
define $(x;m) \approx (y;p)$  in $V_\Lambda$ if and only if
\begin{enumerate}
\item[(V1)] $x(m\wedge d(x))=y(p\wedge d(y))$ and
\item[(V2)] $m-m\wedge d(x) = p- p \wedge d(y)$.
\end{enumerate}
It is straightforward to check that $\approx$ is an equivalence relation, and we denote the class of $(x;n)$ by $[x;n]$.
We view $V_{\Lambda}/\!\!\approx$ as a set of vertices that includes
a copy of the vertices in $\Lambda^0$. Indeed, if $v \in \Lambda^0$, then $v=x(0)$ for some $x \in \Lambda^{\leq \infty}$; we can think of $\Lambda^0$ with the equivalence classes of elements $(x;0)$, and item (V1) above  ensures
this identification is well defined.   
The set $V_\Lambda/\!\!\approx$ also includes elements that are not identified
with vertices in $\Lambda^0$.  
In particular, for each $x \in \Lambda^{\leq \infty}$ 
with $d(x) < \infty$, there is one element in $V_\Lambda/\!\!\approx$ for each  $m > d(x)$. 
  
Second, define $(x;(m,n))\sim (y;(p,q))$ in $P_\Lambda$ if and only if
\begin{enumerate}
\item[(P1)] $x(m \wedge d(x),n \wedge d(x))=y(p \wedge d(y), q \wedge d(y))$,
\item[(P2)] $m-m \wedge d(x) = p - p \wedge d(y)$ and
\item[(P3)] $n - m = q - p$.
\end{enumerate}
Again, it is straightforward to check that $\sim$ is an equivalence relation, and  we denote the class of $(x;(m,n))$ by $[x;(m,n)]$. The next proposition is \cite[Proposition~4.9]{W}, and says that we can view each $[x;(m,n)] \in P_{\Lambda}/\!\!\sim$ 
as a morphism between vertices $[(x;n)]$ and $[(x;m)]$ in $V_{\Lambda}/\!\!\approx$ of degree $n-m$.

\begin{prop}[Farthing, Webster] Let $\Lambda$ be a locally convex, row-finite 
$k$-graph.
Define
\begin{align*}
&\tL^0 := V_\Lambda/\!\!\approx \text{\ and\ }\tL := P_{\Lambda}/\!\!\sim, \\
&r,s:\tL \to \tL^0  \text{ by } 
r([x;(m,n)]) = [x;m] \text{ and }
 s([x;(m,n)]) = [x;n],\\
&\operatorname{id}([x;m])=[x;(m,m)],\\
&[x;(m,n)] \circ [y;(p,q)] = [x(0,n \wedge d(x))\sigma^{p \wedge d(y)}(y);(m,n+p-q)],
\text{ and}\\
&d:\tL \to \N^k \text{ by } 
d(v) = 0 \text{ for all } v \in \tL^0 \text{ and }
d([x;(m,n)]) = n-m. 
\end{align*}
Each of these functions is well defined, and 
 $\tL=(\tL^0, \tL, r, s, \operatorname{id},\circ, d)$  is a $k$-graph without sources.  We call $\tL$ the
\emph{desourcification} of $\Lambda$.  
\end{prop}

We have decided to use $d,r$ and $s$ to denote the appropriate maps in 
both $\Lambda$ and $\tL$.  This doesn't seem to cause any confusion.   The following remark and lemmas
 give insight into the structure of $\tilde\Lambda$.

\begin{rmk}\label{rmk-structurelambda0}
\begin{enumerate}
\item\label{it:rmk6.2a} The map
$\iota:\Lambda \to \tL$ defined by $\iota(\lambda) = [\lambda x;(0,d(\lambda))]$
for $x \in s(\lambda)\Lambda^{\leq \infty}$ is a well-defined, injective $k$-graph 
morphism  \cite[Proposition~4.13]{W}.  Notice that if $v \in \Lambda^0$, then  $\iota(v) = [x; (0,0)]$ for some $x \in v\Lambda^{\leq \infty}$. 
\item The map
$\pi:\tL \to \iota(\Lambda)$  defined by  
$\pi([y;(m,n)] = [y; (m \wedge d(y), n \wedge d(y))]$ 
is a well-defined, surjective $k$-graph morphism  such that $\pi \circ \pi = \pi$ and
$\pi \circ \iota=\iota$ \cite[page~168]{W}.
\item\label{item-c} If $\mu, \lambda \in \iota(\Lambda^0)\tL$ such 
that $d(\lambda)=d(\mu)$ and $\pi(\lambda)=\pi(\mu)$, then $\lambda=\mu$  \cite[Lemma~4.19]{W}.
\end{enumerate}
\end{rmk}

\begin{lemma}\label{lem:extending_bd_paths}
 Suppose $\mu := [z; (0,n)] \in \tL$.  Then $\mu \in \iota(\Lambda)$ if and only if 
$d(z) \geq n$.
\end{lemma}

\begin{proof}
Suppose $\mu = [z;(0,n)] = \iota(\lambda)$ for some $\lambda \in \Lambda$.  Since $\iota$ is degree-preserving, $d(\lambda) = n$. Since $\pi\circ \iota=\iota$ we have $\pi(\mu)=[z;(0,n\wedge d(z))]=\iota(\lambda)$. Thus $n\wedge d(z)=n$, that is, $d(z)\geq n$.

Conversely, suppose $d(z)\geq n$. Then factor $z=\lambda y$ where $\lambda\in\Lambda^n$ and $y\in\Lambda^{\leq\infty}$. Now $\mu=[z;(0,n)]=[\lambda y;(0,d(\lambda))]=\iota(\lambda)$. 
\end{proof}

The following straightforward lemma is very 
useful and  is used without proof in \cite{W}.
\begin{lemma}
\label{lem:0dlambda}
Suppose $\lambda \in \tL$ such that
 $r(\lambda) \in \iota(\Lambda^0)$.  Then there exists 
$x \in \Lambda^{\leq \infty}$ such that 
$\lambda = [x;(0,d(\lambda))]$.
\end{lemma}
 \begin{proof}
  Write $\lambda = [z;(m,n)]$ for some $m,n \in \N^k$ and $z\in \Lambda^{\leq \infty}$.
If $m=0$, then $d(\lambda)=n$ and we are done.  So suppose $m > 0$.  Because
 $[z;m]=r(\lambda) \in \iota(\Lambda^0)$,
$[z;m] \approx [y;0]$ for some $y \in \Lambda^{\leq \infty}$.  By (V2) $m-m \wedge d(z) = 0$. 
Then $m = m \wedge d(z)$, and hence $m \leq d(z)$.  Let $x:=\sigma^m(z)$. 
It suffices to show that $(x;(0,n-m)) \sim (z;(m,n))$.  
Items (P2) and (P3) are obvious.  To see (P1), notice
\begin{align*}
 z(m \wedge d(z),n \wedge d(z)) &= z(m,n \wedge d(z))= \sigma^m(z)(0, n \wedge d(z)-m)\\
&= x(0, n \wedge d(z)-m)= x(0, (n-m) \wedge d(x))
\end{align*}
as needed.
 \end{proof}

\begin{lemma}\label{lem:Claire_13_14}
Suppose $v \in \iota(\Lambda^0)$ and  $\lambda \in v \tilde{\Lambda}^n$ with
$\pi(\lambda) = v$.  If 
$\mu \in v\tilde{\Lambda}^{n}$,
then $\lambda=\mu$. 
\end{lemma}

\begin{proof} By \cite[Lemma~4.19]{W} (see Remark~\ref{rmk-structurelambda0}(\ref{it:rmk6.2a})),
it suffices to show that $\pi(\mu) = \pi(\lambda)$.  
Since $r(\lambda) = r(\mu) = v \in \iota(\Lambda^0)$, 
by Lemma~\ref{lem:0dlambda} 
there exist $x,y \in \Lambda^{\leq \infty}$ such that
\[
 \lambda = [y;(0,n)],
\mu = [x;(0,n)] \text{ and }
[x,0]=[y,0] = v.\]
By definition of $\pi$,
\[ \pi(\lambda) = [y; (0, n \wedge d(y))] \text{ and }\\
\pi(\mu) = [x;(0, n \wedge d(x))].
\]
By assumption, $\pi(\lambda) = v$, and hence $n \wedge d(y) = 0$.
We need to show that $n \wedge d(x) = 0$ as well; we do this by showing that if $n_i\neq 0$ then $d(x)_i=0$.
Suppose $n_i\neq 0$. Since $n \wedge d(y) = 0$ we have $d(y)_i=0$. 
Since $y$ is a boundary path,
$0 \Omega_{k,d(y)}^{\leq e_i} = \{0\}$.  
Now condition (\ref{eq:boundarypath}) gives  
$y(0)\Lambda^{\leq e_i} = \{y(0)\}$.   Hence $y(0)\Lambda^{e_i}=v\Lambda^{e_i}=\emptyset$.  But $x(0)=v$ and so $d(x)_i=0$ as well. Thus $n\wedge d(x)=0$.
Now $\pi(\lambda)=v=\pi(\mu)$ and $\lambda,\mu$ have the same degree, so $\lambda=\mu$ by  \cite[Lemma~4.19]{W}.
\end{proof}

\begin{lemma}\label{lem-minextensions}
Let $\lambda,\mu\in\Lambda$. Then $\tilde\Lambda^{\min}(\iota(\lambda),\iota(\mu))=\iota(\Lambda)^{\min}(\iota(\lambda),\iota(\mu))$.
\end{lemma}

\begin{proof}
 This is essentially \cite[Lemma~4.22]{W}.  Since
$\iota(\Lambda) \subseteq \tL$, we have \[\iota(\Lambda)^{\min}(\iota(\lambda),\iota(\mu))
 \subseteq \tilde\Lambda^{\min}(\iota(\lambda),\iota(\mu)).\]  Let $(\alpha, \beta) \in \tilde\Lambda^{\min}(\iota(\lambda),\iota(\mu))$.
Then
\[\iota(\lambda)\alpha = \iota(\mu)\beta = [z;(0,d(\lambda) \vee d(\mu))]\]
for some $z \in \Lambda^{\leq \infty}$ by Lemma~\ref{lem:0dlambda}.  
Now\[
    (\iota(\lambda)\alpha)(0,d(\lambda)) = \iota(\lambda) = [z;(0,d(\lambda))]
   \]
by the factorisation property in $\tL$.  So $d(z) \geq d(\lambda)$
by Lemma~\ref{lem:extending_bd_paths}.  Similarly, $d(z) \geq d(\mu)$.
Hence $d(z) \geq d(\lambda) \vee d(\mu)$ and so $\iota(\lambda)\alpha \in \iota(\Lambda)$
by Lemma~\ref{lem:0dlambda} again.
Now by the unique factorisation property, both $\alpha$ and $\beta$ are in $\iota(\Lambda)$.
Thus \[\tilde\Lambda^{\min}(\iota(\lambda),\iota(\mu)) \subseteq \iota(\Lambda)^{\min}(\iota(\lambda),\iota(\mu))\]
as needed.
\end{proof}


\section{The Morita equivalence of  $\KP_R(\Lambda)$ and  $\KP_R(\tL)$}\label{sec-Morita}
Throughout this section, $\Lambda$ is a locally convex, row-finite $k$-graph and $R$ is a commutative ring with $1$.
We start by identifying $\KP_R(\Lambda)$ with a subalgebra of $\KP_R(\tL)$.
\begin{prop}
\label{prop:Claire_Prop_11''}
Let $\Lambda$ be a locally convex, row-finite $k$-graph and let $\tilde{\Lambda}$ be its 
desourcification. Let $(q,t)$ be  a universal  $\KP_R(\tilde{\Lambda})$-family.
Let $B(\tilde\Lambda)$ be the subalgebra  of $\KP_R(\tilde{\Lambda})$ generated by $\{q_{\iota(v)},
 t_{\iota(\lambda)}, t_{\iota(\mu^*)}:v\in\Lambda^0, \lambda,\mu\in\Lambda\}$. There is a  graded isomorphism of  $\KP_R(\Lambda)$ onto $B(\tilde\Lambda)$. 
\end{prop}

\begin{proof}
We show that $(q \circ \iota,t \circ \iota)$ is a Kumjian-Pask $\Lambda$-family in  $\KP_R(\tilde{\Lambda})$; 
we start by verifying (KP1), (KP2) and (KP4$'$), and then we use Corollary~\ref{cor-altKP3'} to verify (KP3$'$).

Since $(q,t)$ is a Kumjian-Pask $\tilde{\Lambda}$-family, 
$\{q_v:v \in \tilde{\Lambda}^0\}$ is an orthogonal set of idempotents, and hence so is
$\{q_{\iota(w)}:w \in \Lambda^0\}$. This gives (KP1) for $(q \circ \iota,t \circ \iota)$.

Let $\lambda, \mu \in \Lambda^{\neq 0}$ with $r(\mu)=r(\lambda)$. Since 
(KP2) holds for $(q,t)$ 
and $\iota$ is a graph morphism, 
$t_{\iota(\lambda)} t_{\iota(\mu)} =t_{\iota(\lambda)\iota(\mu)}
= t_{\iota(\lambda\mu)}$.
Similarly, the other equations in (KP2) hold for $(q \circ \iota,t \circ \iota)$.

For (KP4$'$), it suffices by Lemma~\ref{lem-kp4alt} to show that for $1\leq i\leq k$ with
$v\Lambda^{e_i} \neq \emptyset$, 
\begin{equation*}\label{eq-kp4alt}q_{\iota(v)} = \sum_{\lambda \in v\Lambda^{e_i}} t_{\iota(\lambda)}t_{\iota(\lambda)^*}.
\end{equation*}
Suppose $v\Lambda^{e_i}\neq \emptyset$, and let $\mu \in \iota(v)\tL^{e_i}$.  Then $\mu=[x;(0,e_i)]$ for some $x \in v\Lambda^{\leq \infty}$ by Lemma~\ref{lem:0dlambda}. Since $v\Lambda^{e_i}\neq \emptyset$ and $r(x)=v$, we have $d(x)\geq e_i$. Thus $\mu=\iota(\lambda)$ for some $\lambda\in v\Lambda^{e_i}$ by Lemma~\ref{lem:extending_bd_paths}, and this $\lambda$ is unique because $\iota$ is injective.
Now
\begin{align*}
q_{\iota(v)}&= \sum_{\mu \in \iota(v) \tL^{e_i}} t_\mu t_{\mu^*} 
\quad \text{by (KP4) for $(q,t)$,}\\
&= \sum_{\lambda \in v\Lambda^{e_i}} t_{\iota(\lambda)}t_{\iota(\lambda)^*}.
\end{align*}
Thus (KP4$'$) holds for $(q \circ \iota,t \circ \iota)$.

For (KP3$'$), let  $\lambda,\mu\in\Lambda$. Then 
\begin{align*}
t_{\iota(\lambda)^{*}}t_{\iota(\mu)}=
\sum_{(\sigma,\tau)\in\tilde\Lambda^{\min}(\iota(\lambda),\iota(\mu))}
t_{\sigma} t_{\tau^*} \text{\quad by \eqref{eq-span}}\\
=\sum_{(\alpha,\beta)\in\Lambda^{\min}(\lambda,\mu)}
t_{\iota(\alpha)} t_{\iota(\beta)^*} \text{\quad by Lemma~\ref{lem-minextensions}.}
\end{align*}
By Corollary~\ref{cor-altKP3'}, (KP3$'$) holds for $(q \circ \iota,t \circ \iota)$, and hence
$(q \circ \iota,t \circ \iota)$ is a Kumjian-Pask $\Lambda$-family in $\KP_R(\tilde\Lambda)$.

Let $(p,s)$ be a generating Kumjian-Pask $\Lambda$-family in $\KP_R(\Lambda)$.  Since $(q \circ \iota,t \circ \iota)$ 
is a $\Lambda$-family in $\KP_R(\tL)$, the universal property of $\KP_R(\Lambda)$ (Theorem~\ref{thm-KP}\eqref{item-a}) gives an $R$-algebra homomorphism  $\pi_{q \circ \iota,t \circ \iota}:\KP_R(\Lambda)\to \KP_R(\tilde\Lambda)$ such that $\pi_{q \circ \iota,t \circ \iota}(p_v)=q_{\iota(v)}$,  $\pi_{q \circ \iota,t \circ \iota}(s_\lambda)=t_{\iota(\lambda)}$ and $\pi_{q \circ \iota,t \circ \iota}(s_{\lambda^*})=t_{\iota(\lambda)^*}$.
Since $\iota$ is degree preserving, $\pi_{q \circ \iota,t \circ \iota}$ is graded.  The graded-uniqueness 
theorem (Theorem~\ref{gut})
implies  $\pi_{q \circ \iota,t \circ \iota}$ is injective.  Since  $\{q \circ \iota(v), t \circ \iota(\lambda), 
t \circ \iota(\mu^*)\}$ generates $B(\tilde\Lambda)$, the range of $\pi_{q \circ \iota,t \circ \iota}$ is $B(\tilde\Lambda)$.
\end{proof}

\begin{prop}
\label{lem:KPiotaL}
Let $(\Lambda,d)$ be a locally convex, row-finite $k$-graph and let $(\tilde{\Lambda}, d)$ be its
desourcification. Let $(q,t)$ be  a generating Kumjian-Pask $\Lambda$-family in $\KP_R(\tilde{\Lambda})$ and $B(\tilde\Lambda)$ 
be the subalgebra  of $\KP_R(\tilde{\Lambda})$ generated by $\{q_{\iota(v)}, t_{\iota(\lambda)}, t_{\iota(\mu^*)}:v\in\Lambda^0, \lambda,\mu\in\Lambda\}$.
Then
\[B(\tilde\Lambda) = \lsp\{t_{\alpha}t_{\beta^*} : \alpha, \beta \in \tilde{\Lambda},
 r(\alpha),r(\beta) \in \iota(\Lambda^0)\}.\]
\end{prop}

\begin{proof}
The $\subseteq$ direction is obvious.  To see the other containment,  consider $t_{\alpha}t_{\beta^*}$ where
$\alpha, \beta \in \tilde{\Lambda}, r(\alpha),r(\beta) \in \iota(\Lambda^0)$.  
We may assume $s(\alpha)=s(\beta)$ for otherwise $t_{\alpha}t_{\beta^*} = 0$ by (KP2) and (KP1).  
Since $r(\alpha),r(\beta) \in \iota(\Lambda^0)$, 
there exist $x,y \in \Lambda^{\leq \infty}$ such that
\[
 \alpha = [x;(0,d(\alpha))] \text{ and }
\beta = [y;(0,d(\beta))]
\] by Lemma~\ref{lem:0dlambda}.
Using the definition of $\pi$ and the factorisation property in $\tL$, there exist
$\gamma,\gamma' \in \tL$ such that
$\alpha = \pi(\alpha)\gamma$ and $\beta=\pi(\beta)\gamma'$. Write  
$\gamma = [x;(d(\alpha) \wedge d(x), d(\alpha))]$ and 
$\gamma' = [y; (d(\beta) \wedge d(y),d(\beta))]$.

We claim that $\gamma = \gamma'.$ 
To prove the claim, we show that  items (P1)--(P3) hold for $(x;(d(\alpha) \wedge d(x), d(\alpha)))$ 
and $(y; (d(\beta) \wedge d(y),d(\beta)))$.
First notice that (P2) is trivial; both the left and right-hand
 sides of the (P2) equation are 0.
 Now (V1) implies that
\begin{equation}
 \label{eq:alpha_beta_v1}
x(d(\alpha) \wedge d(x)) = y(d(\beta) \wedge d(y))
\end{equation}
and (V2) implies that
\begin{equation}
 \label{eq:alphabeta_v2}
d(\alpha)-d(\alpha) \wedge d(x) = d(\beta) -d(\beta) \wedge d(y).
\end{equation}
Notice that equation (\ref{eq:alphabeta_v2}) is precisely 
(P3), and (P1) follows immediately from (\ref{eq:alpha_beta_v1}).  Thus $\gamma=\gamma'$
 as claimed.  

Now we have 
\begin{equation}\label{eq:talphatbeta*}
t_{\alpha}t_{\beta^*} = t_{\pi(\alpha)}t_{\gamma}t_{\gamma^*}t_{\pi(\beta)^*}.
\end{equation}
Our next claim is  that
$t_{\gamma}t_{\gamma^*} = q_{r(\gamma)}$.
Since $\tL$ has no sources, (KP4) says  that 
\[q_{r(\gamma) } = \sum_{\delta \in r(\gamma)  \tL^{d(\gamma)}} t_{\delta}t_{\delta^*}.\]
But $\pi(\gamma)=r(\gamma)$  so  
$r(\gamma)\tL^{d(\gamma)} = \{\gamma\}$
by Lemma~\ref{lem:Claire_13_14}.
Thus $q_{r(\gamma)} = t_{\gamma}t_{\gamma^*}$ as claimed. 
Finally, from~\eqref{eq:talphatbeta*}
we have 
$t_{\alpha}t_{\beta^*} = t_{\pi(\alpha)}q_{r(\gamma)} t_{\pi(\beta^*)} =  
t_{\pi(\alpha)} t_{\pi(\beta^*)}\in B(\tilde\Lambda)$  since the range of $\pi$ is $\iota(\Lambda)$. 
\end{proof}

We are ready to show $\KP_R(\Lambda)$ and $\KP_R(\tL)$ are Morita equivalent.  First,  consider
how the analogous proof  proceeds in the $C^*$-setting in \cite{W} and \cite{F}. 
Let $A$ be a $C^*$-algebra and $p$ a projection in the multiplier algebra $M(A)$ of $A$. Then $pAp$ 
is a sub $C^*$-algebra of $A$ and $\overline{ApA}$ is an ideal of $A$, and $pA$ is a $pAp$--$\overline{ApA}$ 
imprimitivity bimodule , giving a $pAp$--$\overline{ApA}$ Morita equivalence \cite[Example~2.12]{tfb}.  
Both  Farthing and Webster show that the $C^*$-algebra $C^*(\Lambda)$ of a $k$-graph $\Lambda$ is
 Morita equivalent to the $C^*$-algebra $C^*(\tilde\Lambda)$ of the  desourcification $\tilde\Lambda$, 
that $\sum_{v \in \iota(\Lambda^0)} q_v$ converges to a projection $p$ in $M(C^*(\tilde\Lambda))$,
 that $p$ is full in the sense that $C^*(\tilde\Lambda)pC^*(\tilde\Lambda)=C^*(\tilde\Lambda)$, and 
identify $pC^*(\tilde\Lambda)p$ with $C^*(\Lambda)$.  The work required to do this in the algebraic setting of Kumjian-Pask algebras is similar: the sum  $\sum_{v \in \iota(\Lambda^0)} q_v$  may not add up to an idempotent $p$ in $\KP_R(\tilde\Lambda)$, but we can write down  
analogues of $p\KP_R(\tilde\Lambda)$ and $\KP_R(\tilde\Lambda)p$ without explicit reference to $p$.  
If $\iota(\Lambda^0)$ is finite, then $p=\sum_{v \in \iota(\Lambda^0)} q_v$ is defined in $\KP_R(\tilde\Lambda)$, and $\KP_R(\tilde\Lambda)p=\lsp\{t_\lambda t_\mu^*:\lambda,\nu\in\tilde\Lambda, r(\mu)\in\iota(\Lambda^0)\}$; notice that the right-hand-side still makes sense when $\iota(\Lambda^0)$ is infinite.

\begin{lemma}\label{lem-helpMorita} The subset \[M := \lsp\{t_{\lambda}t_{\mu^*} : \lambda, \mu \in \tilde{\Lambda}, r(\mu)\in \iota(\Lambda^0)\}\] of $\KP_R(\tilde\Lambda)$ is closed under multiplication on the left by $\KP_R(\tilde\Lambda)$ and on the right by $B(\tilde\Lambda)$.  The subset \[N := \lsp\{t_{\lambda}t_{\mu^*} : \lambda, \mu \in \tilde{\Lambda}, r(\lambda)\in \iota(\Lambda^0)\}\] of $\KP_R(\tilde\Lambda)$ is closed under multiplication on the left by $B(\tilde\Lambda)$ and on the right by $\KP_R(\tilde\Lambda)$.  Further,
\begin{gather*}
MN:=\lsp\{mn:m\in M, n\in N\}=\KP_R(\tilde\Lambda)\text{\ and\ } \\NM:=\lsp\{nm:m\in M, n\in N\}=B(\tilde\Lambda).
\end{gather*}
\end{lemma}

\begin{proof}  Let $t_\alpha t_{\beta^*}\in \KP_R(\tilde\Lambda)$, $t_{\lambda}t_{\mu^*}\in M$  and $t_\eta t_{\xi^*}\in B(\tilde\Lambda)$. Let $a=d(\beta)\vee d(\lambda)$ and $b=d(\mu)\vee d(\eta)$. By Proposition~\ref{prop:lspfamily},
\[
t_\alpha t_{\beta^*}t_{\lambda}t_{\mu^*}=\sum_{d(\beta\sigma)=a,\beta\sigma=\lambda\tau} t_{\alpha\sigma}t_{(\mu\tau)^*}\in M
\]
because $r(\mu\tau)=r(\mu)\in \iota(\Lambda^0)$. Similarly, 
\[
t_{\lambda}t_{\mu^*}t_\eta t_{\xi^*}=\sum_{d(\mu\sigma)=b,\mu\sigma=\eta\tau}t_{\lambda\sigma}t_{(\xi\tau)^*}\in M
\]
because $r(\xi\tau)=r(\xi)\in \iota(\Lambda^0)$. 
 Thus $M$ is closed under multiplication on the left by $\KP_R(\tilde\Lambda)$ and on the right by $B(\tilde\Lambda)$. The analogous assertion about $N$ follows from  Proposition~\ref{prop:lspfamily} in the same way.

Since $M, N\subseteq \KP_R(\tilde\Lambda)$, so is $MN$. For the other inclusion, consider $t_{\alpha}t_{\beta^*}\in \KP_R(\tilde\Lambda)$. We may assume that $t_{\alpha}t_{\beta^*}\neq 0$, and  then $s(\alpha)=s(\beta)$. First suppose that $s(\alpha)\in \iota(\Lambda^0)$. 
Then $t_{\alpha}t_{\beta^*}=t_\alpha q_{s(\alpha)} t_{\beta^*}=t_\alpha t_{s(\alpha)^*}t_{s(\alpha)} t_{\beta^*} \in MN$. 
Second, suppose that $s(\alpha) \notin \iota(\Lambda^0)$.  Then $s(\alpha)=[x;n]$ for some $n \in \N^k$
and $x \in \Lambda^{\leq \infty}$. Thus $\lambda:=[x;(0,n)] \in \tL$ and $r(\lambda) = [x;0]\in \iota(\Lambda^0)$. 
Then (KP3) gives 
\[
t_\alpha t_{\beta^*} = t_\alpha q_{s(\alpha)} t_{\beta^*}=t_\alpha t_{\lambda^*} t_\lambda t_{\beta^*} \in MN.
\]
Thus $MN=\KP_R(\tilde\Lambda)$.
                                                       
Next consider $n=t_{\gamma}t_{\delta^*}\in N$ and  $m=t_{\lambda}t_{\mu^*}\in M$. Let $c=d(\mu)\vee d(\gamma)$.  By Proposition~\ref{prop:lspfamily},
\[
nm=\sum_{d(\delta\sigma)=c,\delta\sigma=\lambda\tau} t_{\gamma\sigma}t_{(\mu\tau)^*}.
\]
In each summand, $r(\gamma\sigma)=r(\gamma)\in\iota(\Lambda^0)$ and $r(\mu\tau)=r(\mu)\in\iota(\Lambda^0)$, and hence each summand is in $B(\tilde\Lambda)$ by Proposition~\ref{lem:KPiotaL}. Thus $NM\subseteq B(\tilde\Lambda)$. For the other inclusion, consider $t_\eta t_{\xi^*}\in B(\tilde\Lambda)$. Then $t_\eta t_{\xi^*}=t_\eta q_{s(\eta)}q_{s(\xi)}t_{\xi^*}=t_\eta t_{s(\eta)}t_{s(\xi)}t_{\xi^*}\in NM$. Thus $NM=B(\tilde\Lambda)$.
\end{proof}

Now we recall the notion of a  Morita context from \cite[page~41]{GS}. Let $A$ and $B$ be rings, $M$ an $A$--$B$ bimodule, $N$ a $B$--$A$ bimodule, and 
\[
     \psi:M \otimes_B N \to A \text{ and }\phi:N \otimes_A M \to B
\]
 bimodule homomorphisms satisfying
\begin{equation}\label{eq-ops}n'\cdot \psi(m \otimes n) = \phi(n' \otimes m)\cdot n \text { and } m'\cdot \phi(n \otimes m) = \psi(
m' \otimes n)\cdot m\end{equation}  for  $n,n' \in N$ and $m,m' \in M$.
Then  $(A,B,M,N,\psi, \phi)$ is a \emph{Morita context} between $A$ and $B$; it is called \emph{surjective} if $\psi$ and $\phi$ are surjective.

\begin{thm}\label{thm:m_context}
Let $(\Lambda,d)$ be a locally convex, row-finite $k$-graph and let $(\tilde{\Lambda}, d)$ its
desourcification. Let $(q,t)$ be  a universal  $\KP_R(\tilde{\Lambda})$-family and  $B(\tilde\Lambda)$ be the subalgebra  of $\KP_R(\tilde{\Lambda})$ generated by $\{q_{\iota(v)}, t_{\iota(\lambda)}, t_{\iota(\mu^*)}:v\in\Lambda^0, \lambda,\mu\in\Lambda\}$. 
Let $M, N$ be as in Lemma~\ref{lem-helpMorita}. Then
\begin{enumerate}
\item\label{item-aa}
$M$ is a $\KP_R(\tL)$-$B(\tilde\Lambda)$ bimodule and $N$ is a $B(\tilde\Lambda)$-$\KP_R(\tL)$
bimodule with the module actions given by multiplication in $\KP_R(\tL)$;
\item\label{item-bb} there are surjective maps
$\psi: M \otimes_{B(\tilde\Lambda)} N \to \KP_R(\tilde{\Lambda})$ and $\phi: N \otimes_{\KP_R(\tilde{\Lambda})} M \to \KP_R(\Lambda)$ such that  
$\psi(m \otimes_{B(\tilde\Lambda)} n ) = mn$ and $\phi(n \otimes_{\KP_R(\tilde{\Lambda})} m ) = nm$ for $n\in N$ and $m\in M$; 
\item\label{item-cc} $(\KP_R(\tilde{\Lambda}), B(\tilde\Lambda), M, N, \psi, \phi)$ is a surjective Morita context between
$\KP_R(\tL)$ and $B(\tilde\Lambda)$.
\end{enumerate}
\end{thm}

Composing with the isomorphism $\pi_{q \circ \iota,t \circ \iota}:\KP_R(\Lambda)\to B(\tilde\Lambda)$ of Proposition~\ref{prop:Claire_Prop_11''} gives: 

\begin{cor}\label{cor-Morita} Let $(\Lambda,d)$ be a locally convex, row-finite $k$-graph and let $(\tilde{\Lambda}, d)$ be the 
desourcification of $\Lambda$. Then there is a surjective Morita context between $\KP_R(\tL)$ and $\KP_R(\Lambda)$.
\end{cor}

\begin{proof}[Proof of Theorem~\ref{thm:m_context}]
Lemma~\ref{lem-helpMorita} gives \eqref{item-aa}. For \eqref{item-bb} we start by observing that the map $f:M\times N\to\KP_R(\tL)$ is bilinear, so that by the universal property of the tensor product there is a unique linear map $\psi_f:M\otimes N\to \KP_R(\tL)$ such that $\psi_f(m\otimes n)=mn$. The range of $\psi_f$ is $MN$, which is $\KP_R(\tL)$ by Lemma~\ref{lem-helpMorita}. To see that $\psi_f$ factors through the quotient map $q:M\otimes N\to  M \otimes_{B(\tilde\Lambda)} N$, we observe that for $x\in {B(\tilde\Lambda)}$ we have
$
\psi_f(m\cdot x\otimes n-m\otimes x\cdot n)=(mx)n-m(xn)=0.
$
Now there is a unique linear map $\psi:M \otimes_{B(\tilde\Lambda)} N \to \KP_R(\tilde{\Lambda})$ such that $\psi\circ q=\psi_f$. Thus $\psi$ is surjective and $\psi(m \otimes_{B(\tilde\Lambda)} n ) = mn$. That $\psi$ is a bimodule homomorphism follows because the actions are by multiplication.
The analogous assertions about $\phi$ follows in the same way. This gives \eqref{item-bb}.

For \eqref{item-cc} it remains to verify \eqref{eq-ops}, but this is immediate since everything is defined in terms of the associative multiplication in $\KP_R(\tL)$.
\end{proof}
  
Let $E$ be a row-finite directed graph, and $F$ the directed graph obtained from $E$ by adding ``infinite heads'' to sources, as in \cite[page~310]{BPRS}.  Then the path categories $\Lambda_E$ and $\Lambda_F$ are row-finite $1$-graphs; $E$ may have sources but is trivially locally convex.  
By \cite[Proposition~4.11]{W}, ${\tilde\Lambda}_E$ is isomorphic to $\Lambda_F$. 
The Leavitt path algebra $L_R(F)$ is the universal $R$-algebra generated by a Leavitt $F$-family \cite[Definition~3.1]{T}, and any Leavitt $F$-family gives a Kumjian-Pask $\Lambda_F$-family and vice versa.   Thus by Theorem~\ref{thm-KP}\eqref{item-a}, $L_R(F)$ and $\KP_R(\Lambda_F)$ are isomorphic. Similarly, $L_R(E)$ and $\KP_R(\Lambda_E)$ are isomorphic. Thus we obtain the following corollary which is the analogue of the $C^*$-algebraic result \cite[Lemma~1.2(c)]{BPRS}.

\begin{cor}
Let $E$ be a row-finite directed graph, and $F$ the directed graph obtained from $E$ by adding ``infinite heads'' to sources.  Then there is a surjective Morita context between the Leavitt path algebras $L_R(F)$ and $L_R(E)$.
\end{cor}

There is a more general procedure, called ``desingularisation'', which takes a directed graph $E$ 
that is not necessarily row-finite and constructs a row-finite directed graph $F$ without sources such that $C^*(F)$ and $C^*(E)$ are Morita equivalent \cite[Theorem~2.11]{DT}.


\section{Simplicity and basic simplicity}\label{sec-simplicity}

Throughout this section, $\Lambda$ is a locally convex, row-finite $k$-graph,
$\tL$ is its desourcification and $R$ is a commutative ring with $1$.  
We show
that the Morita context between $\KP_R(\tL)$ and $\KP_R(\Lambda)$ of Corollary~\ref{cor-Morita}
preserves  \emph{basic} ideals (see below for the definition).   We then transfer simplicity 
results about $\KP_R(\tL)$ proved in \cite{ACaHR} to $\KP_R(\Lambda)$. 

A subset $H \subseteq \Lambda^0$ is called \emph{hereditary} if
for every $v \in H$ and $\lambda \in \Lambda$ with $r(\lambda)=v$
we have $s(\lambda) \in H$.   We say $H$ is  \emph{saturated} if for $v \in \Lambda^0$,
\[s(v\Lambda^{\leq e_i}) \subseteq H \text{ for some } i \in \{1, ..., k\} \implies
 v \in H.
\] See \cite[page~113]{RSY03}.
In this section, we will apply these definitions to $\tL$ which has no 
sources; then the  definition of `saturated' above is equivalent to: or $v \in \Lambda^0$,
\[s(v\tL^{n}) \subseteq H \text{ for some } n \in \N^k \implies
 v \in H.
\] 
\begin{lemma}
\label{lem:sat_pi}
Let $\Lambda$ be a locally convex, row-finite $k$-graph and $\tL$ 
its desourcification.  Suppose that $H$ is a saturated and hereditary subset of $\tL^0$.  Then
$v \in H$ if and only if $\pi(v) \in H$.  
\end{lemma}

\begin{proof}
Let $v \in \tL^0$.  Write  
$v = [z;m]$ where $z \in \Lambda^{\leq \infty}$ and
$m \in \N^k$.  Then $\pi(v)=[z;d(z) \wedge m]$. 
Notice that $\lambda := [z; (d(z),m \wedge d(z))] \in \tL$ has source $v$ and range $\pi(v)$
and  \[\pi(v)\tL^{m-m\wedge d(z)} = \{\lambda\}\] by
Lemma~\ref{lem:Claire_13_14}.

Now suppose $v \in H$.    Since \[
                  s\big(\pi(v)\tL^{m-m\wedge d(z)}\big) = \{s(\lambda)\} = \{v\} \subseteq H,
                 \]
and $H$ is saturated, $\pi(v) \in H$ as well.
Conversely, suppose $\pi(v) \in H$.  Then $r(\lambda) = \pi(v)\in H$
and  $H$  hereditary implies $v=s(\lambda) \in H$.
\end{proof}

Following \cite{T}, we say an ideal $I$ in $\KP_R(\Lambda)$ is a \emph{basic ideal} if  $rp_v \in I$ for some $v\in\Lambda^0$ and $r \in R\setminus\{0\}$   
imply $p_v \in I$.
We say $\KP_R(\Lambda)$ is \emph{basically simple} if its only basic ideals are $\{0\}$ and $\KP_R(\Lambda)$.

The Morita context of Theorem~\ref{thm:m_context} induces a lattice isomorphism $L$ between the ideals
of $\KP_R(\tL)$ and ideals of $B(\tL)$  such that $L(I)=NIM$ \cite[Proposition~3.5]{GS}.  

\begin{prop}
\label{prop:basicideals} Let $(\KP_R(\tL),B(\tL),M,N)$ be the Morita context 
 of Theorem~\ref{thm:m_context} and $L:I\mapsto NIM$ be 
be the induced lattice isomorphism from the ideals of $\KP_R(\tL)$ to the ideals of 
$B(\tilde\Lambda)$.  
Then $I$ is a basic ideal in $\KP_R(\tL)$ if and only if $L(I)$ is a basic ideal in $B(\tL)$.
\end{prop}

\begin{proof} 
Suppose $I$ is a  basic ideal of
$\KP_R(\tL)$.   Fix $v \in \iota(\Lambda^0)$ and nonzero $r \in R$ such that $rp_v \in L(I)$.
Since
$N$ and $M$ are subsets of $\KP_R(\tL)$ and $I$ is an ideal of $\KP_R(\tL)$, we have $L(I)=NIM \subseteq I$.  Thus
$rp_v \in I$ and hence $p_v \in I$ because $I$ is a basic ideal.
Since $v \in \iota(\Lambda^0)$, we have  $p_v \in M \cap N$.  Now
\[
 p_v=p_vp_vp_v \in NIM = L(I).
\]
Thus $L(I)$ is a basic ideal of $B(\tL)$. 

Conversely, let $I$ be an ideal in $\KP_R(\tL)$ such that  $L(I)$ is a basic ideal of $B(\tL)$.  We will
show $I$ is a basic ideal.
Fix $v \in \tL^0$ and nonzero $r \in R$ such that $rp_v \in I$.
  Then $v$ is an element of 
\[H_{I,r} := \{v \in \tL^0 : rp_v \in I\},\]
which is saturated and hereditary by \cite[Lemma~5.2]{ACaHR}.  Hence $\pi(v) \in H_{I,r}$ by Lemma~\ref{lem:sat_pi} which means $rp_{\pi(v)} \in I$.
Notice that $\pi(v) \in \iota(\Lambda^0)$ and so $p_{\pi(v)} \in M \cap N$.   Then
\[rp_{\pi(v)} = p_{\pi(v)}rp_{\pi(v)}p_{\pi(v)} \in NIM = L(I).
\]
Therefore $p_{\pi(v)} \in L(I)$ because $L(I)$ is basic.
Thus we have \[p_{\pi(v)} = p_{\pi(v)}p_{\pi(v)}p_{\pi(v)}  \in ML(I)N = I.\]
Now  $\pi(v) \in H_{I,1}$ which is  saturated and hereditary by \cite[Lemma~5.2]{ACaHR}.
Therefore $v \in H_{I,1}$ by Lemma~\ref{lem:sat_pi}; that is $p_v \in I$.
\end{proof}

Combining Proposition~\ref{prop:basicideals} with the isomorphism of
$\KP_R(\Lambda)$ onto
$B(\tL)$ of Proposition~\ref{prop:Claire_Prop_11''}, the
proof of the following corollary is immediate.

\begin{cor}
\label{cor:basicsimple}
Let $\Lambda$ be a locally convex, row-finite $k$-graph and $\tL$ its 
desourcification.  Then $\KP_R(\Lambda)$ is basically simple if and only
if $\KP_R(\tL)$ is basically simple.
\end{cor}

Next we want to transfer results about $\tL$ to $\Lambda$.  To use results 
already in the literature we need to reconcile some (of the many) aperiodicity conditions that have been used.

\begin{lemma}\label{lem-pain} Let $\Lambda$ be a  row-finite $k$-graph.\begin{enumerate} 
\item\label{paina} Suppose that $\Lambda$ has no sources. The following aperiodicity conditions  are equivalent:
\begin{enumerate}

\item\label{itemii} \textup{(``Condition B'' from \cite[Theorem~4.3]{RSY03} for graphs without sources; our \eqref{eq:aperiodic} reduces to this)}\\
For every $v\in\Lambda^0$ there exists $x\in v\Lambda^\infty$ such that $\mu\neq \nu\in \Lambda$ implies $\mu x\neq \nu x$.
\item \textup{(The ``no local periodicity condition'' from \cite[Lemma~3.2(iii)]{RS2007})}\\
For every $v\in \Lambda^0$ and all $m\neq n\in\N^k$, there exists $x\in v\Lambda^\infty$ such that $\sigma^m(x)\neq\sigma^n(x)$.

\item\label{itemiii} \textup{(The finite-path reformulation from \cite[Lemma~3.2(iv)]{RS2007}; used in \cite{ACaHR})}\\ For every $v\in \Lambda^0$ and $m\neq n\in {\mathbb N}^k$ there exists $\lambda\in v\Lambda$ such that
$d(\lambda)\geq m\vee n$ and 
$
\lambda(m,m+d(\lambda)-(m\vee n))\neq \lambda(n,n+d(\lambda)-(m\vee n))
$.
\end{enumerate}
\item\label{painb}
Suppose that $\Lambda$ is locally convex. The following aperiodicity conditions on $\Lambda$ are equivalent:
\begin{enumerate}
\item \textup{(``Condition B''  from \cite[Theorem~4.3]{RSY03}, this is our \eqref{eq:aperiodic})}\\
For every $v \in \Lambda^0$, there exists  $x \in v\Lambda^{\leq \infty}$
such that $\alpha \neq \beta\in\Lambda$  implies $\alpha x \neq \beta x$.
\item \textup{(The ``no local periodicity condition'' from \cite[Definition~3.2]{RS2009})}\\
For every $v\in\Lambda^0$ and all $m\neq n\in\N^k$, there exists $x\in\Lambda^{\leq \infty}$ such that either $m-m\wedge d(x)\neq n-n\wedge d(x)$ or $\sigma^{m\wedge d(x)}(x)\neq \sigma^{n\wedge d(x)}(x)$. 
\end{enumerate}
\end{enumerate}
\end{lemma}

\begin{proof} For the equivalences in \eqref{paina}, see \cite[Lemma~3.2]{RS2009}, and for the equivalences in \eqref{painb}, see
\cite[Proposition~2.11]{Shotwell}. 
\end{proof}

A locally convex, row-finite graph $k$-graph $\Lambda$ is said to be \emph{cofinal} in \cite[Definition~3.1]{RS2009}
if for every $x\in\Lambda^{\leq\infty}$ and every $v\in \Lambda^0$, there exists $n\in\N^k$  such that $n\leq d(x)$ and
 $v\Lambda x(n)\neq \emptyset$. When $\Lambda$ has no sources, the cofinality condition reduces 
to the one given in section~\ref{sec_cofinal}.

\begin{thm}
 Let $\Lambda$ be a locally convex, row-finite  $k$-graph and $R$ be a commutative ring with $1$.  Then
\begin{enumerate}
\item\label{item-bs} $\KP_R(\Lambda)$ is basically simple if and only if  $\Lambda$ is cofinal and aperiodic;

\item\label{item-s} $\KP_R(\Lambda)$ is simple if and only if  $\Lambda$ is cofinal and aperiodic, and $R$ is a field.
\end{enumerate}
\end{thm}

\begin{proof}  \eqref{item-bs}
By Corollary~\ref{cor:basicsimple}, $\KP_R(\Lambda)$ is basically simple if and only if $\KP_R(\tL)$ is basically simple. 
 Since $\tL$ is row-finite with no sources, by \cite[Theorem~5.14]{ACaHR}, $\KP_R(\tL)$ is basically simple if and only if $\tL$ is cofinal and aperiodic in the sense of condition \eqref{itemiii} of Lemma~\ref{lem-pain}.

By \cite[Proposition~3.5]{RS2009}, $\tL$ is cofinal if and only if $\Lambda$ is cofinal. By \cite[Proposition~3.6]{RS2009}, $\tL$ has
 no local periodicity  if and only if $\Lambda$ has no local periodicity.
Using   Lemma~\ref{lem-pain} it follows that $\KP_R(\Lambda)$ is basically simple if and only if $\Lambda$ is cofinal and aperiodic 
(in the sense of the equivalent conditions of Lemma~\ref{lem-pain}\eqref{painb}).

\eqref{item-s} By Corollary~\ref{cor-Morita}, $\KP_R(\Lambda)$ is Morita equivalent to  $\KP_R(\tL)$. Thus $\KP_R(\Lambda)$ is simple if and only if $\KP_R(\tL)$ is. By \cite[Theorem~6.1]{ACaHR}, $\KP_R(\tL)$ is simple if and only if $R$ is a field and $\tL$ is cofinal and aperiodic.  The same shenanigans as in \eqref{item-bs} now give that $\KP_R(\Lambda)$ is simple if and only if $R$ is a field and $\Lambda$ is cofinal and aperiodic.
\end{proof}


\section{The ideal structure}\label{sec-ideals}
Throughout this section $\Lambda$ is a locally convex, row-finite $k$-graph, and  $\tL$ is its desourcification. Recall that $B(\tilde\Lambda)$ is the subalgebra  of $\KP_R(\tilde{\Lambda})$ generated by $\{q_{\iota(v)}, t_{\iota(\lambda)}, t_{\iota(\mu)}:v\in\Lambda^0, \lambda,\mu\in\Lambda\}$, and is canonically isomorphic to $\KP_R(\Lambda)$ by Proposition~\ref{prop:Claire_Prop_11''}.  We now use the Morita context between $\KP_R(\tL)$ and $B(\tL)$ to show that there is a lattice isomorphism from the hereditary, saturated subsets of $\Lambda^{0}$ onto the basic, graded ideals of $\KP_R(\Lambda)$. This extends \cite[Theorem~5.1]{ACaHR} to locally convex, row-finite $k$-graphs.

\begin{lemma}\label{inversepi}
Let $\pi:\tL\to\iota(\Lambda)$ be the projection defined in Remark~\ref{rmk-structurelambda0}. Then $H\mapsto  \pi(H)$ is a lattice isomorphism of the hereditary, saturated subsets of $\tL^0$ onto the hereditary, saturated subsets of $\iota(\Lambda^0)$.  
\end{lemma}

\begin{proof}
Let $H$ be a hereditary, saturated subset of $\tL^0$. To see that $\pi(H)$ is a hereditary subset of $\iota(\Lambda^0)$, let $v\in\pi(H)$ and suppose $\lambda\in v\iota(\Lambda)$. Since $r(\lambda)\in\pi(H)$ we have $r(\lambda)\in H$ by Lemma~\ref{lem:sat_pi}. Then  $s(\lambda)\in H$ because $H$ is hereditary. Since $\lambda\in\iota(\Lambda)$ and $\pi$ is a graph morphism, we have $s(\lambda)=s(\pi(\lambda))=\pi(s(\lambda))\in\pi(H)$. Thus $\pi(H)$ is hereditary.

To see that $\pi(H)$ is  saturated, let $v\in \iota(\Lambda^0)$, and suppose that 
$s(v\iota(\Lambda^{\leq e_i}))\subseteq \pi(H)$ for some $1\leq i\leq k$.  There are two cases. First, suppose that $v\Lambda^{\leq e_i}=\{v\}$. Then $\{v\}=s(v\iota(\Lambda^{\leq e_i}))\subseteq \pi(H)$ gives $v\in\pi(H)$ as required.  Second, suppose that $v\Lambda^{\leq e_i}=v\Lambda^{e_i}$. Then $v\iota(\Lambda^{\leq e_i})=v\tL^{e_i}$. Also $\pi(H)\subseteq H$ by Lemma~\ref{lem:sat_pi}. Thus $s(v\tilde\Lambda^{e_i})\subseteq H$, and since $H$ is saturated in $\tilde\Lambda^0$, we get $v\in H$.  Now $v=\pi(v)\in\pi(H)$ as required.  Thus $\pi(H)$ is a saturated subset of $\iota(\Lambda)$.

To see that $H\mapsto\pi(H)$ is injective, suppose $\pi(H)=\pi(K)$. Let $v\in H$. Then  $\pi(v)\in \pi(H)=\pi(K)$, and hence $v\in K$ by Lemma~\ref{lem:sat_pi}. Thus $H\subseteq K$, and the other set inclusion follows by symmetry. Thus $H=K$, and $H\mapsto\pi(H)$ is injective.

To see that $H\mapsto\pi(H)$ is onto, let $G$ be a hereditary, saturated subset of $\iota(\Lambda^0)$.  Since $\pi(\pi^{-1}(G))=G$, it suffices to show that $\pi^{-1}(G)$ is a hereditary, saturated subset of $\tilde\Lambda^0$. Let $v\in \pi^{-1}(G)$ and suppose $\lambda\in v\tL$. Then $r(\pi(\lambda))=\pi(r(\lambda))=\pi(v)\in G$.  Since $G$ is hereditary, $s(\pi(\lambda))=\pi(s(\lambda))\in G$.  Thus $s(\lambda)\in \pi^{-1}(G)$, and hence $\pi^{-1}(G)$ is hereditary. To see $\pi^{-1}(G)$ is saturated, let $v\in\tL^0$, and suppose $s(v\tilde\Lambda^{e_i})\subseteq \pi^{-1}(G)$ for some $1\leq i\leq k$.  
Then
\[
s(\pi(v)\iota(\Lambda^{\leq e_i}))=s(\pi(v)\pi(\tilde\Lambda^{e_i}))=\pi(s(v\tilde\Lambda^{e_i}))\subseteq G.
\]
But $G$ is saturated in $\iota(\Lambda^0)$, and hence $\pi(v)\in G$. Now $v\in\pi^{-1}(G)$ as needed, and $\pi^{-1}(G)$ is saturated.  It follows that $H\mapsto\pi(H)$ is onto.

Finally, $H\mapsto\pi(H)$ is a lattice isomorphism because $H_1\subseteq H_2$ if and only if $\pi(H_1)\subseteq \pi(H_2)$. 
\end{proof}

The next lemma is a generalisation of \cite[Lemma~5.4]{ACaHR} to graphs with possible sources.

\begin{lemma}\label{ideal-gen}  Let $G$ be a hereditary, saturated subset of $\Lambda^0$, and $J_G$ be the ideal of $\KP_R(\Lambda)$ generated by $\{p_v:v\in G\}$. 
 Then 
 \begin{equation}J_G=\lsp\{s_\alpha s_{\beta^*}:\alpha,\beta\in\Lambda, s(\alpha)=s(\beta)\in G\}.\label{eq-is-ideal}\end{equation}
\end{lemma}

\begin{proof} Denote the right-hand side of \eqref{eq-is-ideal} by $J$. For $v\in G$, taking $\alpha=\beta=v$ shows $p_v\in J$. Thus $J_G$ is contained in the ideal generated by $J$.   But each $s_\alpha s_{\beta^*}\in J$ is in $J_G$ because $s_\alpha s_{\beta^*}=s_\alpha p_{s(\alpha)}s_{\beta^*}$.  Thus \eqref{eq-is-ideal}  will follow if  $J$ is an ideal.  To see this, let $s_\mu s_{\nu^*}\in\KP_R(\Lambda)$ and $s_\alpha s_{\beta^*}\in J$ such that $s_\mu s_{\nu^*}s_\alpha s_{\beta^*}\neq 0$. Then $r(\alpha)=r(\nu)$ and by Corollary~\ref{cor-altKP3'}, $s_\mu s_{\nu^*}s_\alpha s_{\beta^*}=\sum_{(\gamma,\delta)\in\Lambda^{\min}(\nu,\alpha)}s_{\mu\gamma}s_{(\beta\delta)^*}$.  For any nonzero summand, $r(\delta)=s(\alpha)\in H$ implies $s(\delta)=s(\gamma)\in G$ since $G$ is hereditary.  Thus each $s_{\mu\gamma}s_{(\beta\delta)^*}\in J$, and hence $s_\mu s_{\nu^*}s_\alpha s_{\beta^*}\in J$ as well.  Similarly $s_\alpha s_{\beta^*}s_\mu s_{\nu^*}\in J$, and it follows that $J$ is an ideal. 
\end{proof}

\begin{prop}\label{MOrita-graded}  
Let  $L$ be 
the lattice isomorphism from the ideals of $\KP_R(\tL)$ to the ideals of 
$B(\tilde\Lambda)$ induced by the Morita context of   Theorem~\ref{thm:m_context}. 
\begin{enumerate} \item\label{MOrita-graded-a} Then $I$ is a graded ideal of $\KP_R(\tL)$ if and only if $L(I)$ is a graded ideal of  $B(\tilde\Lambda)$.
\item\label{MOrita-graded-b} Let $H$ be a hereditary, saturated subset of $\tL^0$. Then
$L(I_H)=J_{\pi(H)}.$
\end{enumerate}
\end{prop}
\begin{proof}
\eqref{MOrita-graded-a} Here $\KP_R(\tilde\Lambda)$ and $B(\tL)$ are graded by the subgroups
\begin{gather*}\KP_R(\tilde\Lambda)_j:= \lsp\{t_\alpha t_{\beta^{*}}: \alpha,\beta\in\tL, d(\alpha)-d(\beta)=j\}\text{\ and }\\
B(\tL)_j:= \lsp\{t_\alpha t_{\beta^{*}}: \alpha,\beta\in\iota(\Lambda), d(\alpha)-d(\beta)=j\}=B(\tL)\cap \KP_R(\tilde\Lambda)_j,
\end{gather*}
respectively. Let  $m\in M$ and $n\in N$. Since $M$ and $N$ are submodules of $\KP_R(\tilde\Lambda)$,  we can write  $m=\sum m_i$ and $n=\sum n_j$ where each $m_j, n_j\in \KP_R(\tilde\Lambda)_j$.  

Now suppose that $I$ is a graded ideal of $\KP_R(\tilde\Lambda)$.  To check that $L(I)=NIM$ is graded it suffices to check that every element of $L(I)$ is a sum of elements in $\cup_j(L(I)\cap B(\tL)_j)$. Let $x\in I$, and write $x=\sum x_l$ where each $x_l\in I\cap \KP_R(\tilde\Lambda)_l$.  For an  element $y=nxm\in L(I)$ we have
$
y=\sum_{i,j,l} m_ix_ln_j;
$
each $m_ix_ln_j\in I\cap \KP_R(\tilde\Lambda)_{i+j+l}= I\cap B(\tL)_{i+j+l}$. Thus $L(I)$ is a graded ideal of $B(\tL)$.  The other direction follows in the same way.

\eqref{MOrita-graded-b}  By \cite[Lemma~5.4]{ACaHR}, 
$
I_H=\lsp\{t_\alpha t_{\beta^{*}}: \alpha,\beta\in\tL, s(\alpha)=s(\beta)\in H\}
$, and then
\begin{align}
L(I_H)&=NI_HM\notag\\
&=\lsp\{t_\lambda t_{\mu^*}t_\alpha t_{\beta^{*}}t_\sigma t_{\tau^{*}}: \lambda, \mu, \alpha,\beta, \sigma,\tau\in\tL,s(\alpha)=s(\beta)\in H, r(\lambda),r(\tau)\in\iota(\Lambda^0)\}.\label{eq-mess}
\end{align}
Consider a term $t_\lambda t_{\mu^*}t_\alpha t_{\beta^{*}}t_\sigma t_{\tau^{*}}$ as in \eqref{eq-mess}.  By three applications of Corollary~\ref{cor-altKP3'},
\begin{align*}
t_\lambda (t_{\mu^*}t_\alpha)( t_{\beta^{*}}t_\sigma) t_{\tau^{*}}&=\sum_{(\gamma,\delta)\in\Lambda^{\min}(\mu,\alpha)}
\ \sum_{(\xi,\eta)\in\Lambda^{\min}(\beta,\sigma)} t_\lambda t_\gamma (t_{\delta^*}t_\xi) t_{\eta^*}t_{\tau^{*}}\\
&=\sum_{(\gamma,\delta)\in\Lambda^{\min}(\mu,\alpha)}
\ \sum_{(\xi,\eta)\in\Lambda^{\min}(\beta,\sigma)} \ \sum_{(\rho,\epsilon)\in\Lambda^{\min}(\delta,\xi)}t_\lambda t_\gamma t_{\rho}t_{\epsilon^*} t_{\eta^*}t_{\tau^{*}}
\end{align*}
For each summand   $t_\lambda t_\gamma t_{\rho}t_{\epsilon^*} t_{\eta^*}t_{\tau^{*}}$, we have $r(\delta\rho)=r(\xi\epsilon)=s(\alpha)\in H$. Since $H$ is hereditary, $s(\rho)=s(\epsilon)\in H$. Thus, from \eqref{eq-mess}, 
\[
L(I_H)=\lsp\{t_\lambda t_{\tau^*}:\lambda,\tau\in\tL, r(\lambda),r(\tau)\in\iota(\Lambda^0), s(\lambda)=s(\tau)\in H\}
\]
(the above shows the inclusion $\subseteq$, and the reverse inclusion is trivial).  But $B(\tilde\Lambda)$ 
is the subalgebra  of $\KP_R(\tilde{\Lambda})$ generated by $\{q_{\iota(v)}, t_{\iota(\mu)}, t_{\iota(\nu)^*}:v\in\Lambda^0, \lambda,\mu\in\Lambda\}$. Thus 
\[
L(I_H)=\lsp\{t_\lambda t_{\tau^*}:\lambda,\tau\in\iota(\Lambda), s(\lambda)=s(\tau)\in H\cap \iota(\Lambda^0)\}.
\]
Using Lemma~\ref{lem:sat_pi}, we have $H\cap \iota(\Lambda^0)=\pi(H)$, and by Lemma~\ref{inversepi}, $\pi(H)$ is a hereditary subset of $\iota(\Lambda^0)$.  Thus $L(I_H)=J_{\pi(H)}$ by Lemma~\ref{ideal-gen}.
\end{proof}

\begin{thm}\label{thm-ideals}
Let $\Lambda$ be a locally convex, row-finite $k$-graph. Then $G\mapsto J_G$ is a lattice isomorphism from the hereditary, saturated subsets of $\Lambda^0$ to the basic graded ideals of $\KP_R(\Lambda)$.
\end{thm}
\begin{proof}
 Since $\tilde\Lambda$ is row-finite without sources, $H\mapsto I_H$ is a lattice isomorphism from the hereditary, saturated subsets of $\tL^0$ onto the basic, graded ideals of $\KP_R(\tL)$ by \cite[Theorem~5.1]{ACaHR}.   By Theorem~\ref{thm:m_context}, $\KP_R(\tL)$ is Morita equivalent to its subalgebra $B(\tilde\Lambda)$, and the induced lattice isomorphism  sends $I_H$ to $J_{\pi(H)}$ by Proposition~\ref{MOrita-graded}\eqref{MOrita-graded-b}.  By  Proposition~\ref{prop:basicideals} and Proposition~\ref{MOrita-graded}\eqref{MOrita-graded-a}, respectively, the Morita equivalence preserves basic, graded ideals.  Thus $I_H\mapsto J_{\pi(H)}$  maps onto the basic, graded ideals of $B(\tilde\Lambda)$.
 The canonical isomorphism of $\KP_R(\Lambda)$ onto $B(\tilde\Lambda)$ of Proposition~\ref{prop:Claire_Prop_11''} is graded, and hence the ideal $J_{\iota^{-1}(\pi(H))}$ of  $\KP_R(\Lambda)$ corresponding to $J_{\pi(H)}$ is basic and graded, and all basic, graded ideals of $\KP_R(\Lambda)$ arise this way.
 Composing with the lattice isomorphism $G\mapsto \pi^{-1}(G)$ from the hereditary, saturated subsets of $\Lambda^0$  to those of $\tL^0$ of Lemma~\ref{inversepi} gives the lattice isomorphism 
 \[
 G\mapsto \pi^{-1}(G)\mapsto I_{\pi^{-1}(G)}\mapsto J_{\iota(G)}\mapsto J_G.\qedhere
 \]
\end{proof}

\end{document}